\documentclass{siamart171218}

\RequirePackage{amsmath}
\RequirePackage{amsfonts}
\RequirePackage{amssymb}
\RequirePackage{multirow}
\RequirePackage{algorithm}
\RequirePackage{algpseudocode}
\RequirePackage{graphicx}
\RequirePackage{xcolor}
\allowdisplaybreaks

\newcommand{\tauktrial}{\tau_{k,\text{trial}}}

\newcommand{\finf}{f_{\text{inf}}}
\newcommand{\fsup}{f_{\text{sup}}}

\newcommand{\Phibar}{\bar\Phi}
\newcommand{\alphamin}{\alpha_{\min}}
\newcommand{\taumin}{\tau_{\min}}
\newcommand{\sigmamin}{\sigma_{\min}}

\newcommand{\lipg}{L_{g}}
\newcommand{\lipJ}{L_{J}}
\newcommand{\kappabar}{\bar\kappa}
\newcommand{\kappag}{\kappa_{\nabla f}}
\newcommand{\kappagr}{\kappa_{\partial r}}

\newcommand{\kappac}{\kappa_{c}}
\newcommand{\kappav}{\kappa_{v}}

\newcommand{\kappaJ}{\kappa_J}

\newcommand{\kappayz}{\kappa_{yz}}
\newcommand{\kappaPhi}{\kappa_{\Phi}}

\newcommand{\taubar}{\bar\tau}

\newcommand{\kappabarPhi}{\bar\kappa_{\Phi}}
\newcommand{\alphabar}{\bar\alpha}
\newcommand{\alphabarmin}{\bar\alpha_{\min}}
\newcommand{\xstar}{x_*}

\newcommand{\chibar}{\bar\chi}
\newcommand{\fus}{f_{\text{Algorithm~\ref{alg:main}}}}
\newcommand{\fBazinga}{f_{\text{Bazinga}}}

\newcommand{\nlim}{{n_{\Lcal}}}

\newcommand{\epsminL}{\epsilon^\Lcal_{\min}}
\newcommand{\sigmaminL}{\sigma^\Lcal_{\min}}
\newcommand{\deltamin}{\delta_{\min}}
\newcommand{\deltaminL}{\delta_{\min}^\Lcal}

\usepackage{lmodern}
\usepackage{pifont}
\usepackage{wrapfig}

\newtheorem{assumption}{Assumption}[section]

\usepackage{dpr_fec,ifthen}
\usepackage{hyperref}

\newif\iftechreport
\techreportfalse

\title{A Proximal-Gradient Method for Solving Regularized Optimization Problems with General Constraints\footnotemark[1]}
\author{Frank E. Curtis\footnotemark[2]
   \and Xiaoyi Qu\footnotemark[2]
   \and Daniel P.~Robinson\footnotemark[2]}
\date  {\today}

\begin{document}

\maketitle

\renewcommand{\thefootnote}{\fnsymbol{footnote}}
\footnotetext[1]{This material is based upon work supported by the U.S. National Science Foundation, Division of Mathematical Sciences, Computational Mathematics Program under Award Number DMS--1016291.}
\footnotetext[2]{Department of Industrial and Systems Engineering, Lehigh University, Bethlehem, PA, USA; E-mail: \email{\{frank.e.curtis, xiq322, daniel.p.robinson\}@lehigh.edu}}
\renewcommand{\thefootnote}{\arabic{footnote}}

\begin{abstract}
We propose, analyze, and test a proximal-gradient method for solving regularized optimization problems with general constraints. The method employs a decomposition strategy to compute trial steps and uses a merit function to determine step acceptance or rejection. Under various assumptions, we establish a worst-case iteration complexity result, prove that limit points are first-order KKT points, and show that manifold identification and active-set identification properties hold. Preliminary numerical experiments on a subset of the CUTEst test problems and sparse canonical correlation analysis problems demonstrate the promising performance of our approach.
\end{abstract}

\begin{keywords}
proximal-gradient method,
nonlinear optimization, nonconvex optimization, worst-case iteration complexity, regularization, composite optimization, constrained optimization
\end{keywords}

\begin{AMS}
  49M37, 65K05, 65K10, 65Y20, 68Q25, 90C30, 90C60
\end{AMS}

\numberwithin{equation}{section}
\numberwithin{theorem}{section}

\section{Introduction}\label{sec:introduction}
We consider the constrained optimization problem
\begin{equation}~\label{prob:general}
\min_{x\in\R{n}} \ f(x) + r(x) \ \ \text{subject to (s.t.)}  \ \ c(x) = 0, \ x \in \Omega,
\end{equation}
where $f:\R{n} \to \R{}$ is continuously differentiable, $r: \R{n} \to [0,\infty)$ is a nonnegative-valued convex function (possibly nonsmooth),   $c:\R{n} \to \R{m}$ is continuously differentiable with $m \leq n$, and $\Omega$ is the nonnegative orthant in $\R{n}$ (i.e., the vectors in~$\R{n}$ with all nonnegative components). We note that general inequality constraints can be converted to the form~\eqref{prob:general} by using slack variables. Thus, problem~\eqref{prob:general} is important to a range of application areas such as data science (e.g., principal component analysis~\cite{zou2018selective} and canonical correlation analysis~\cite{witten2009penalized,yang2019survey}), finance (e.g., portfolio selection~\cite{bai2016splitting,cui2013convex}), signal processing (e.g., sparse blind deconvolution~\cite{zhang2017global} and array beamformer design~\cite{hamza2019hybrid,huang2023sparse}), and image processing (e.g., hyperspectral unmixing~\cite{chouzenoux2020proximal}).

When the constraints in~\eqref{prob:general} are not present, the problem reduces to a nonsmooth unconstrained regularized optimization problem, for which proximal-gradient (PG) methods and their variants are among the most widely used algorithms~\cite{BecT09,beck2017first,CheCR18,CheCR17,KariNutiSchm16,lee2019inexact}. 
The basic PG method proceeds by  solving a sequence of proximal subproblems. Given the $k$th iterate $x_k \in \R{n}$ and  proximal parameter $\alpha_k > 0$, the next iterate $x_{k+1}$ is computed as the unique solution to the optimization problem
\begin{equation}
\min_{x \in \mathbb{R}^n} \left\{ \tfrac{1}{2\alpha_k} \|x - (x_k - \nabla f(x_k))\|_2^2 + r(x) \right\}.
\label{eq:prox_def}
\end{equation}
A notable property of PG methods is that as $\alpha_k \to 0$, the vector $x_{k+1} - x_k$ converges to zero. PG methods are also well-known for their \emph{structure identification} property~\cite{lee2023accelerating, nutini2019active,sun2019we}, whereby the sequence of iterates eventually identifies the manifold associated with a solution (e.g., the zero-nonzero structure of an optimal solution when $r(x) = \|x\|_1$). This property is particularly advantageous in structured optimization problems for at least three reasons.  First, identifying the correct solution structure can have significant computational savings.  For example, when $r(x) = \|x\|_1$, it is well known that optimal solutions tend to be sparser, and in the context of statistical modeling sparser solutions offer simpler models that can be employed more efficiently~\cite{han2015learning,hoefler2021sparsity}.  Second, in certain other applications, the zero-nonzero values of the variables can have a physical meaning that is lost if the solutions do not have the true zero-nonzero structure~\cite{dorfler2014sparsity,fardad2011sparsity,tonneau2020sl1m}.  Third, if the manifold of the solution can be identified, then one can consider hybrid methods that combine PG calculations with those of more advanced (usually higher-order) optimization algorithms designed for \emph{smooth} optimization problems (here restricted to the smooth manifold identified by the PG iterates). Such an approach aims to exploit local smoothness to achieve accelerated convergence rates, and has great success in many settings~\cite{bareilles2023newton, lee2023accelerating,liu2025proximal}. 

When the regularization function $r$ is not present in problem~\eqref{prob:general}, it reduces to a traditional nonlinear program. An important concept in the nonlinear programming literature is \emph{active-set identification}.  An algorithm has the active-set identification property if, under certain reasonable assumptions, it can identify from an iterate near an optimal solution which inequality constraints are active (i.e., hold at equality) at that optimal solution. For a comprehensive overview of active-set identification strategies in nonlinear programming, see~\cite{facchinei1998accurate,oberlin2006active} and the references therein.

Little research has considered the case when the regularization function $r$ and nonlinear constraints are present. Two primary challenges arise in this setting. First, the computation of projections onto the feasible points satisfying $c(x) = 0$ (or perhaps the intersection of this region with $\Omega$) is typically computationally intractable. Second, conventional techniques such as penalty-based methods~\cite{de2024interior} may fail to preserve the structure of the solution (see~\cite[Section 5]{dai2024proximal}), therefore limiting their effectiveness in this setting. Our work is  motivated by the need to address these challenges.

\subsection{Related work}\label{sec:related-work}
We restrict our attention to work that considers regularized optimization problems with smooth nonlinear constraints, where both the smooth part of the objective and the constraints may be nonconvex. Most approaches are penalty-function-based, where constrained problems are transformed into unconstrained ones (or ones with simple constraints) by combining the objective function with a penalty function that measures constraint violation. The resulting subproblems are then typically solved using the PG method or its variants. Penalty-based methods generally fall into two main categories: augmented Lagrangian methods and penalty-barrier methods. Among these, \cite{bourkhissi2025complexity,li2021rate,sahin2019inexact} propose inexact augmented Lagrangian methods and show that an $\epsilon$-KKT point can be found within $\Ocal(\epsilon^{-3})$ iterations under suitable constraint qualifications. The constraint qualifications in~\cite{li2021rate,sahin2019inexact} are identical, whereas~\cite{bourkhissi2025complexity} uses a slightly different condition, replacing the subdifferential with the horizon subdifferential. In contrast, the augmented Lagrangian method in~\cite{hallak2023adaptive} adopts a transversality condition and establishes a better complexity bound of $\Ocal(\epsilon^{-2})$. In~\cite{de2023constrained}, an augmented Lagrangian method is proposed for solving regularized problems with general constraints.  The authors use an AM-regularity condition to establish convergence, but no complexity result is provided. To the best of our knowledge, \cite{de2024interior} is the only penalty-barrier approach designed for our problem setting.  Instead of assuming any constraint qualification, they directly assume the existence and boundedness of Lagrange multipliers, which is typically implied by a constraint qualification. 

Three non-penalty approaches for solving regularized problems with constraints include~\cite{boob2025level, dai2024proximal,wei2024sqp}.  In~\cite{wei2024sqp}, the authors combine ideas from PG methods and sequential quadratic programming methods. In particular, their method formulates a quadratic approximation to $f$, linearizes the constraint function, and keeps the regularizer explicitly in each subproblem. This nonsmooth subproblem is solved using a semi-smooth Newton method. The weakness of this approach is that each subproblem is assumed to be feasible and no structure identification result is provided. In \cite{boob2025level}, a  feasible proximal-gradient method is proposed that  reformulates a nonconvex problem into convex surrogate subproblems with quadratic regularization, but it cannot handle problems that involve equality constraints due to the infeasibility of each subproblem. Our work builds upon on~\cite{dai2024proximal}, which only considers the equality-constrained case. Although limited in relevance here, we mention that some work has considered problems with only simple bound constraints~\cite{bertocchi2020deep,leconte2024interior} or only  linear constraints~\cite{hajinezhad2019perturbed,jiang2019structured,kong2019complexity}.

\subsection{Contributions}
Our contributions relate to the proposal, analysis, and testing of a new PG algorithm for solving problem~\eqref{prob:general}, as we now discuss. 
\begin{itemize}
\item We propose a new PG method  (Algorithm~\ref{alg:main}) for solving problem~\eqref{prob:general}. Unlike most work in the literature, our method has the following characteristics: (i) it uses the regularization function explicitly (as opposed to approximating it) when computing the trial step, (ii) it avoids using a penalty function to handle the constraints, and (iii) every subproblem is feasible.
\item We establish various convergence results.
(i) Without assuming any constraint qualification, we prove that the number of iterations required to reduce a stationarity measure related to minimizing the constraint violation below $\epsilon > 0$ is $O(\epsilon^{-2})$ (see Theorem~\ref{thm:no-licq}). (ii) Under the linear independence constraint qualification (LICQ), we show that all limit points of the iterate sequence are first-order KKT points (see Theorem~\ref{thm:KKT}). (iii) Under a sequential constraint qualification that is stronger than the LICQ, we prove that the worst-case iteration complexity needed to reduce a KKT measure below $\epsilon > 0$ is $O(\epsilon^{-2})$ (see Theorem~\ref{thm:complexity}). (iv) When strict complementarity holds in addition, we prove that our method possesses an optimal active-set identification property (see Theorem~\ref{thm:active-set-id}). (v) Under partial smoothness of the regularization function $r$ and a certain non-degeneracy assumption, we establish a manifold identification property for our method (see Theorem~\ref{thm:manifold}).
\item We numerically test the performance of our method 
on CUTEst test problems and a sparse canonical correlation analysis problem. In addition, we demonstrate the competitive performance of our algorithm by comparing it  to an augmented Lagrangian approach named Bazinga~\cite{de2023constrained}.
\end{itemize}

\subsection{Organization}
In Section~\ref{sec:preliminary}, we introduce notations and definitions. In Section~\ref{sec:algorithm}, we propose our method as  Algorithm~\ref{alg:main}. In Section~\ref{sec:subproblems}, we derive preliminary results for the subproblems used in our method, which are critical for the theoretical analysis we provide in Section~\ref{sec:analysis}. In Section~\ref{sec:numerical}, we illustrate our algorithm’s performance through numerical tests, and final comments are provided in Section~\ref{sec:conclusion}.

\section{Preliminaries}\label{sec:preliminary}
Let $\R{}$ denote the set of real numbers, $\R{}_{\geq0}$ (resp., $\R{}_{>0}$) denote the set of nonnegative (resp.,~positive) real numbers, $\R{n}$ denote the set of $n$-dimensional real vectors, and $\R{m \times n}$ denote the set of $m$-by-$n$-dimensional real matrices. The set of natural numbers is $\N{} := \{0,1,2,\dots\}$. For a given natural number $n \in \N{}$, let $[n] := \{1,\dots,n\}$. The index sets of active and inactive variables at $x\in\R{n}$ is $\Acal(x) := \{i \in [n]: x_i = 0\}$ and 
$\Ical(x) := \{i \in [n]: x_i \neq 0\}$, respectively. The $\epsilon$-neighborhood ball of a point $x\in \R{n}$ is $\Bcal(x, \epsilon) := \{z \in \R{n} : \|x - z\|_2 < \epsilon \}$.  Given a nonempty set $\Ccal$ that is either compact, or closed and convex, and a point $\xbar\in\R{n}$, the distance from $\xbar$ to $\Ccal$ is $\dist(\xbar,\Ccal) := \min_{x\in\Ccal} \|x-\xbar\|_2$.

For convenience, we define $g(x) :=  \nabla f(x)$ and $J(x) := \nabla c(x)^T$. We append a natural number as a subscript for a quantity to denote its value during an iteration of an algorithm; i.e., we let $f_k := f(x_k)$, $g_k := g(x_k)$, $c_k := c(x_k)$, and $J_k := J(x_k)$. 

We now introduce several key concepts from convex analysis that will be used throughout the paper.  We start with the normal cone~\cite[Theorem 6.9]{rockafellar2009variational}.

\begin{definition}[normal cone]\label{def:normal-cone}
The normal cone of a convex set  $\Ccal$ at $x\in\Ccal$ is
\begin{equation*}
N_{\Ccal}(x) = \{v \in \R{n} : v^T(y-x) \leq 0 \ \ \text{for all} \ \ y \in \Ccal\}.
\end{equation*}
\end{definition}
We define the tangent cone using its polarity with the normal cone~\cite[Theorem 6.28]{rockafellar2009variational}. 
\begin{definition}[tangent cone]\label{def:tangent-cone}
The tangent cone of a convex set $\Ccal$ at $x\in\Ccal$ is 
\begin{equation*}
T_{\Ccal}(x) = \{d \in \R{n} : v^Td \leq 0 \ \ \text{for all} \ \ v \in N_{\Ccal}(x)\}. 
\end{equation*}  
\end{definition}
Next, we define the projection onto a closed convex set~\cite[Proposition 1.1.9]{bertsekas2009convex}.

\begin{definition}[Projection]
Let $\Ccal \subseteq \R{n}$ be a nonempty closed convex set. 
The projection of $x \in \R{n}$ onto $\Ccal$ is $\Proj_{\Ccal}(x) := \arg\min_{y \in \Ccal} \|x-y\|_2$.
\end{definition}

Finally, we define the projection of the steepest descent direction of a  function onto the tangent cone~\cite[Equation (3.1)]{calamai1987projected} associated with $\Omega$ at a point $x$.

\begin{definition}\label{def:projected-steepest-descent}
Given a differentiable function $h:\R{n} \to \R{}$, a convex set $\Ccal$, and $x\in\Ccal$, the projection of the steepest descent direction of $h$ at $x$ onto $T_{\Ccal}(x)$ is
\begin{equation*}
\nabla_{\Ccal} h(x) = \mathop{\arg \min}\limits_{v \in T_{\Ccal}(x)} \ \|v + \nabla h(x)\|_2 \equiv \Proj_{T_\Ccal(x)}(-\nabla h(x)).
\end{equation*}
\end{definition}

\section{Algorithm Framework}\label{sec:algorithm}
The algorithm that we propose for solving problem~\eqref{prob:general} is stated as Algorithm~\ref{alg:main}. Given the $k$th iterate $x_k\in\Omega$, the $k$th proximal parameter $\alpha_k$, and constant $\kappa_v \in \R{}_{>0}$, we first compute a direction $v_k$ that reduces linearized infeasibility within  $\Omega$. In particular, the vector  $v_k$ is computed as an  approximate solution to the bound-constrained trust-region subproblem
\begin{equation}\label{subprob:feasibility}
\min_{v\in\R{n}} \ m_k(v) \ \ \text{s.t.} \ \|v\|_2 \leq \kappa_v \alpha_k \delta_k,\  x_k + v \in\Omega\ \ \text{with}\ \ m_k(v):=\tfrac{1}{2}\|c_k+J_k v\|_2^2, 
\end{equation}
where 
\begin{equation}\label{def:deltak}
\delta_k := \|\nabla_{\Omega} \psi(x_k)\|_2 \equiv \|\Proj_{T_{\Omega}(x_k)}(-J_k^T c_k)\|_2 \ \ \text{with} \ \ \psi(x) := \tfrac{1}{2}\|c(x)\|_2^2.
\end{equation}
If $\delta_k = 0$, then  $v_k \gets 0$ solves~\eqref{subprob:feasibility}.  In this case, if $\|c_k\|_2 \neq 0$, we terminate our algorithm in Line~\ref{line:ifs} since $x_k$ is an infeasible stationary point, i.e., $x_k$ is infeasible for $c(x) = 0$ and is a first-order stationary point for the problem
\begin{equation}\label{prob:isp}
\min_{x\in\Omega} \ \tfrac{1}{2}\|c(x)\|_2^2.
\end{equation}
If $\delta_k \neq 0$, we compute an approximate solution $v_k$ to~\eqref{subprob:feasibility} satisfying
\begin{equation}~\label{eq:vk-condition}
\begin{aligned}
\|v_k\|_2 \leq \kappa_v \alpha_k \delta_k,  \ \ 
x_k + v_k \in \Omega, \ \ \text{and} \ \ m_k(v_k) \leq m_k(v_k^c), 
\end{aligned}
\end{equation}
where $v_k^c$ is a Cauchy point computed using a projected line search along the steepest descent direction of $m_k$ at $v = 0$.  In particular, by defining
\begin{equation}\label{def:vkbeta}
v_k(\beta) \gets \Proj_{\Omega}(x_k - \beta \nabla m_k(0)) - x_k
\equiv \Proj_{\Omega}(x_k - \beta J_k^T c_k) - x_k, 
\end{equation}
we define the Cauchy point as 
\begin{equation}\label{eq:beta-condition}
v_k^c 
:= v_k(\beta_k)  \equiv \text{Proj}_{\Omega}(x_k - \beta_k J_k^T c_k) - x_k 
\end{equation}
where, for some chosen $\gamma\in(0,1)$, 
\begin{equation}\label{def:beta}
\beta_k = \gamma^{i_k}
\end{equation}
with $i_k$ being the smallest nonnegative integer such that $\beta_k$ in~\eqref{def:beta} satisfies
\begin{equation}\label{eq:bound-sufficient-decrease-sub}
  \|v_k(\beta_k)\|_2 \leq \kappa_v \alpha_k \delta_k \ \ \text{and} \ \  m_k(v_k(\beta_k)) \leq m_k(0) + \eta_m \nabla m_k(0)^T v_k(\beta_k)
\end{equation}
for some constant $\eta_m \in (0,1)$.  (It follows from Lemma~\ref{lem:lower-bound-betak} later on that this procedure is well defined.)  Note from the definition of $v_k^c$ (see~\eqref{eq:beta-condition} which ensures $x_k + v_k^c\in\Omega$) and~\eqref{eq:bound-sufficient-decrease-sub} that $v_k^c$ itself satisfies the conditions required of $v_k$ in~\eqref{eq:vk-condition}.


\balgorithm[!th]
\caption{PG method for solving problem~\eqref{prob:general}}
\label{alg:main}
\balgorithmic[1]
   \State \textbf{Input:} $x_0\in\Omega$, $\{\alpha_0,\tau_{-1},\kappa_\tau,\kappa_v\} \subset \R{}_{>0}$, and $\{\xi,\eta_\Phi,\sigma_c,\epsilon_\tau,\gamma,\eta_m\} \subset (0,1)$
   \For{$k = 0,1,2,\dots$}
   \State compute $\delta_k$ in~\eqref{def:deltak}
   \If{$\delta_k = 0$}\label{line:deltak}
       \State set $v_k \gets 0$
       \If{$\|c_k\|_2 \neq 0$}\label{line:ifs.start}
         \State return $x_k$ (infeasible stationary point)\label{line:ifs}
      \EndIf      
   \Else~($\delta_k \neq 0$)  
      \State compute $v_k$ as an approximate solution to~\eqref{subprob:feasibility} satisfying~\eqref{eq:vk-condition} \label{line:vk} 
   \EndIf\label{line:ifs.end}
   \State compute $u_k$ as the unique solution to subproblem~\eqref{subprob:stationarity} \label{line:uk}
   \State set $s_k \gets v_k + u_k$
   \If{$\|s_k\|_2/\alpha_k = 0$}
   \State return $x_k$ (first-order KKT point for problem~\eqref{prob:general}) \label{line:first.order.kkt} 
   \EndIf
   \State compute $\tau_k$ using~\eqref{eq:tau-update}
   \If{$\Phi_{\tau_k}(x_k+s_k) - \Phi_{\tau_k}(x_k) \leq -\eta_\Phi\left(\frac{\tau_k}{4\alpha_k}\|s_k\|_2^2 + \sigma_c (\|c_k\|_2 - \|c_k + J_k s_k\|_2) \right)$}~\label{line:check-suff-decrease}
   \State set $x_{k+1} \to x_k + s_k$ and $\alpha_{k+1} \to \alpha_k$ \label{line:alpha-same} 
   \Else
   \State set $x_{k+1} \to x_k$ and $\alpha_{k+1} \to \xi \alpha_k$
   \EndIf
    \EndFor
    \ealgorithmic
\ealgorithm

Next, we compute a direction $u_k$ that maintains the level of linearized infeasibility achieved by $v_k$ while also reducing a model of the objective function. In particular, we compute $u_k$ as the unique solution to the strongly convex subproblem
\begin{equation}~\label{subprob:stationarity}
\begin{aligned}
\min_{u \in \R{n}} & \ g_k^T u + \tfrac{1}{2\alpha_k} \|u\|_2^2 + \tfrac{1}{\alpha_k}v_k^Tu + r(x_k+v_k+u) \\
\text{s.t.} & \ J_k u = 0, \ \ x_k + v_k + u \in \Omega.
\end{aligned}
\end{equation}
Concerning  subproblem~\eqref{subprob:stationarity}, note that $u=0$ is feasible and that its solution is unique since it is a convex optimization problem with a strongly convex objective function. The overall trial step $s_k$ is defined as $s_k = v_k + u_k$. 

To determine whether the trial step $s_k$ is accepted, we adopt the $\ell_2$ merit function, which for merit parameter $\tau \in\R{}_{>0}$ is defined as 
$$
\Phi_\tau(x) := \tau\big( f(x) + r(x)\big) + \|c(x)\|_2.
$$
During each iteration, the merit parameter is updated so that $s_k$ is a descent direction for the merit function. To ensure that this holds, note that the directional derivative of $\Phi_\tau$ at $x_k$ along $s_k$, denoted as $D_{\Phi_{\tau}}(x_k,s_k)$, satisfies (see~\cite[Lemma 3.3]{dai2024proximal})
\begin{align*}
& \ D_{\Phi_{\tau}}(x_k,s_k) \\
\leq & \ \tau (g_k^T s_k + r(x_k+s_k) - r_k) + \|c_k + J_k s_k\|_2 - \|c_k\|_2 \\
= & -\tfrac{\tau}{2\alpha_k}\|s_k\|_2^2 + \tau \underbrace{(g_k^T s_k + \tfrac{1}{2\alpha_k}\|s_k\|_2^2 + r(x_k+s_k) - r_k)}_{A_k} + \|c_k + J_k s_k\|_2 - \|c_k\|_2.
\end{align*}
Next, for a chosen parameter $\sigma_c \in (0,1)$, we set
\[
\tau_{k,\text{trial}} \gets \left \{\begin{array}{ll}
    \infty & \text{if } A_k \leq 0, \\
    \tfrac{(1-\sigma_c)(\|c_k\|_2 - \|c_k + J_k s_k\|_2)}{g_k^T s_k + \tfrac{1}{2\alpha_k}\|s_k\|_2^2 + r(x_k+s_k) - r_k} & \text{otherwise,} 
\end{array} \right.
\]
and then set, for some chosen $\epsilon_\tau\in(0,1)$, the value of the $k$th merit parameter as
\begin{equation}~\label{eq:tau-update}
\tau_k \gets
\begin{cases}
   \tau_{k-1} & \text{if } \tau_{k-1} \leq \tau_{k,\text{trial}}, \\
   \min\{(1-\epsilon_{\tau})\tau_{k-1},  \tau_{k,\text{trial}}\} & \text{otherwise.} 
\end{cases}
\end{equation}
This merit parameter update strategy ensures that
\begin{equation*}
    D_{\Phi_{\tau_k}}(x_k,s_k) \leq -\tfrac{\tau_k}{2\alpha_k}\|s_k\|_2^2 - \sigma_c (\|c_k\|_2 - \|c_k + J_k s_k\|_2),
\end{equation*}
meaning that the negative directional derivative is lower bounded by critical measures of problem~\eqref{prob:general}. The $k$th iteration is completed by checking whether the merit function achieves sufficient decrease (see Line~\ref{line:check-suff-decrease}), and then defining the next iterate and proximal parameter accordingly. Specifically, if sufficient decrease in the merit function is achieved, the trial step is accepted  (i.e., $x_{k+1} \gets x_k + s_k$) and the proximal parameter value is maintained (i.e., $\alpha_{k+1} \gets \alpha_k$); otherwise, the trial step is rejected (i.e., $x_{k+1} \gets x_k$) and the proximal parameter value is decreased (i.e., $\alpha_{k+1} \gets \xi \alpha_k$ for some $\xi\in(0,1)$).  This update strategy motivates the definition of the index set \begin{equation}\label{def:S}
\Scal := \{k \in \N{}: x_{k+1} = x_k + s_k \},
\end{equation}
which contains the indices of the successful iterations associated with  Algorithm~\ref{alg:main}.

The following assumption is assumed to hold throughout the paper.

\begin{assumption}\label{ass:basic}
Let $\Xcal \subseteq \R{n}$ be an open convex set containing the iterate sequences $\{x_k\}$ and $\{x_k+v_k\}$ generated by Algorithm~\ref{alg:main}. The function $f: \R{n} \rightarrow \R{}$ is bounded over $\Xcal$, and its gradient function $\nabla f: \R{n} \rightarrow \R{}$ is Lipschitz continuous and bounded in norm over $\Xcal$. Similarly, for all $i \in [m]$, the constraint function $c_i: \R{n} \rightarrow \R{}$ is bounded over $\Xcal$, and its gradient function $\nabla c_i: \R{n} \rightarrow \R{}$ is Lipschitz continuous and bounded in norm over $\Xcal$. Finally, the function $r : \R{n} \to \R{}_{\geq 0}$ is convex, and has bounded  subdifferential $\partial r: \R{n} \to \R{n}$ over $\Xcal$.
\end{assumption}

Under Assumption~\ref{ass:basic}, there exist constants $(\finf,\fsup,\kappag,\kappagr,\kappac,\kappaJ,\lipg,\lipJ) \in \R{} \times \R{} \times \R{}_{>0} \times \R{}_{>0} \times \R{}_{>0} \times \R{}_{> 0} \times \R{}_{>0} \times \R{}_{> 0}$ such that for all $x\in\Xcal$ one has
\begin{equation}\label{eq:consequence-of-basic-assumption}
\begin{aligned}
\finf &\leq f(x) \leq \fsup, & \ \ 
\|\nabla f(x)\|_2 &\leq \kappag, & \ \
\|\partial r(x)\|_2 &\leq \kappagr, \\ 
\|c(x)\|_2 &\leq \kappac, & \ \ 
\|\nabla c(x)^T\|_2 &\leq \kappaJ, & 
\end{aligned}
\end{equation}
and for all $(x,\xbar)\in\Xcal\times \Xcal$ one has
\begin{equation}\label{eq.Lipschitz}
  \|\nabla f(x) - \nabla f(\xbar)\|_2 \leq \lipg \|x - \xbar\|_2
  \ \ \text{and}\ \
  \|\nabla c(x)^T - \nabla c(\xbar)^T\|_2 \leq \lipJ \|x - \xbar\|_2.
\end{equation}

\section{Preliminary Properties Related to the Subproblems}\label{sec:subproblems}
In this section, we discuss properties related to the  subproblems used in Algorithm~\ref{alg:main}.  

\subsection{Subproblem (\ref{subprob:feasibility})}\label{sec:trust-region}

In this section, we present properties  related to the computation of the Cauchy point of subproblem~\eqref{subprob:feasibility}, following by a final result related to the computed feasibility steps. Recall that the Cauchy point is defined in~\eqref{eq:beta-condition}.  Our first lemma summarizes properties of $v_k(\cdot)$ (recall~\eqref{def:vkbeta}).

\begin{lemma}\label{lem:vk-monotonicity}
Consider $v_k(\cdot)$  defined in~\eqref{eq:beta-condition}. For all $0 < \beta_2 \leq \beta_1$, it holds that
\begin{subequations}
\begin{align}
\|v_k(\beta_2)\|_2 &\leq
\|v_k(\beta_1)\|_2 \ \ \text{and}  \label{eq:nondecreasing} \\
\|v_k(\beta_1)/\beta_1\|_2 &\leq \|v_k(\beta_2)/\beta_2\|_2. \label{eq:nonincraesing}
\end{align}
\end{subequations}
For all $\beta\in\R{}_{>0}$ it holds that
\begin{subequations}
\begin{align}
-\nabla m_k(0)^T v_k (\beta) &\geq \|v_k(\beta)\|_2^2/\beta \ \ \text{and} \label{eq:ip-bound} \\
\delta_k \equiv \|\nabla_{\Omega} \psi(x_k)\|_2 
&\geq \left\|v_k(\beta)/\beta\right\|_2. \label{eq:lower-bound-ratio}
\end{align}
\end{subequations}
Finally, the following limit holds:
\begin{equation}\label{eq:limiting-vector}
\lim_{\beta \to 0^+} v_k(\beta)/\beta = \nabla_{\Omega}\psi(x_k).
\end{equation}
\end{lemma}
\begin{proof}
Parts~\eqref{eq:nondecreasing}--\eqref{eq:ip-bound} follow from~\cite[Lemma 2]{toint1988global}, part~\eqref{eq:limiting-vector} follows from~\cite[Proposition 2]{mccormick1972gradient}, and part~\eqref{eq:lower-bound-ratio} follows by combining~\eqref{eq:limiting-vector}, \eqref{eq:nonincraesing}, and~\eqref{def:deltak}.
\end{proof}

The next result is a special case of~\cite[Lemma 4.3]{more2006trust}.
\begin{lemma}\label{lem:lower-bound-betak}
Suppose that $\delta_k \neq 0$.  If $\beta \in\R{}_{>0}$ satisfies $m_k(v_k(\beta)) > m_k(0) + \eta_m \nabla m(0)^T v_k(\beta)$, then $\beta \geq (1-\eta_m)/\|J_k^T J_k\|_2$.
\end{lemma}

We now bound the decrease in $m_k$ by using the argument in~\cite[Theorem 4.4]{more2006trust}. 
\begin{lemma}\label{lem:cauchy-decrease-projection}
Suppose that $\delta_k \neq 0$. Then, with respect to the constant $\kappabar_1 := \min\{1, \gamma(1-\eta_m), \gamma\} \equiv \gamma(1-\eta_m) \in (0,1)$, the Cauchy point $v_k^c \equiv v_k(\beta_k)$ satisfies
\begin{equation*}\label{eq:cauchy-decrease-projection}
-\nabla m_k(0)^T v_k(\beta_k) \geq \kappabar_1 \left[\frac{\|v_k(\beta_k)\|_2}{\beta_k}\right] \min \left\{\frac{1}{1+ \|J_k^T J_k\|_2}\left[\frac{\|v_k(\beta_k)\|_2}{\beta_k}\right], \kappa_v \alpha_k \delta_k \right\}.
\end{equation*}
Moreover, with respect to the constant $\kappa_1 := \kappabar_1 \eta_m \equiv \gamma\eta_m(1-\eta_m)\in(0,1)$, it satisfies
\begin{equation*}
m_k(0) - m_k(v_k(\beta_k)) \geq \kappa_1 \left[\frac{\|v_k(\beta_k)\|_2}{\beta_k}\right] \min \left\{\frac{1}{1+ \|J_k^T J_k\|_2}\left[\frac{\|v_k(\beta_k)\|_2}{\beta_k}\right], \kappa_v \alpha_k \delta_k \right\}.
\end{equation*}
\end{lemma}
\begin{proof}
We begin by proving the first inequality by considering three cases. \\
\textbf{Case 1: $\beta_k = 1$.} It follows from~\eqref{eq:ip-bound} and $\beta_k = 1$ that
\begin{align*}
    -\nabla m_k(0)^T v_k(\beta_k) 
    &\geq \beta_k\left[\frac{\|v_k(\beta_k)\|_2}{\beta_k}\right]^2  =  \left[\frac{\|v_k(\beta_k)\|_2}{\beta_k}\right]^2 \\
    &\geq  \frac{\|v_k(\beta_k)\|_2}{\beta_k}\min\left\{\frac{\|v_k(\beta_k)\|_2}{\beta_k},\kappav\alpha_k\delta_k\right\}.
\end{align*}
Combining this result with $1/(1+\|J_k^TJ_k\|_2) \leq 1$ shows that the first inequality holds. \\
\textbf{Case 2: $\beta_k < 1$ and $\|v_k(\gamma^{-1}\beta_k)\|_2 \leq \kappa_v \alpha_k \delta_k$.} Since $\gamma\in(0,1)$, $\|v_k(\gamma^{-1}\beta_k)\|_2 \leq \kappa_v \alpha_k \delta_k$, and the step size  $\gamma^{-1}\beta_k$ was not accepted by the search procedure, the sufficient decrease condition must not have held, i.e., it must hold that $m_k(v_k(\gamma^{-1}\beta_k)) > m_k(0) + \eta_m \nabla m_k(0)^T v_k(\gamma^{-1}\beta_k)$. Combining this inequality with Lemma~\ref{lem:lower-bound-betak} gives
$\gamma^{-1}\beta_k \geq (1-\eta_m)/\|J_k^T J_k\|_2$. Combining this with~\eqref{eq:ip-bound} gives
\begin{align*}
    -\nabla m_k(0)^T v_k(\beta_k) 
    &\geq \beta_k    \left[\frac{\|v_k(\beta_k)\|_2}{\beta_k}\right]^2  \geq \gamma \frac{(1-\eta_m)} {1+\|J_k^T J_k\|_2} \left[\frac{\|v_k(\beta_k)\|_2}{\beta_k}\right]^2 \\
    &\geq \gamma (1-\eta_m) \frac{\|v_k(\beta_k)\|_2}{\beta_k} \min\left\{
    \frac{1}{1+\|J_k^T J_k\|_2}    \left[\frac{\|v_k(\beta_k)\|_2}{\beta_k}\right],\kappav\alpha_k\delta_k\right\}
\end{align*}
so that the first inequality again holds, and completes the proof for this case. \\
\textbf{Case 3: $\beta_k < 1$ and $\|v_k(\gamma^{-1}\beta_k)\|_2 > \kappa_v \alpha_k \delta_k$.} It follows from~\eqref{eq:nonincraesing} and the fact that $\gamma\in(0,1)$ that  $\frac{\|v_k(\beta_k)\|_2}{\beta_k} \geq \frac{\|v_k(\gamma^{-1}\beta_k)\|_2}{\gamma^{-1}\beta_k}$.  After rearrangement and using the fact that  $\|v_k(\gamma^{-1}\beta_k)\|_2 > \kappa_v \alpha_k \delta_k$ in this case, we obtain
$\gamma^{-1}\|v_k(\beta_k)\|_2 \geq \|v_k(\gamma^{-1}\beta_k)\|_2 > \kappa_v \alpha_k \delta_k$, 
which combined with~\eqref{eq:ip-bound} yields
\begin{align*}
    -\nabla m_k(0)^T v_k(\beta_k) 
    &\geq \|v_k(\beta_k)\|_2\left[\frac{\|v_k(\beta_k)\|_2}{\beta_k}\right]  > \gamma \kappa_v \alpha_k \delta_k \left[\frac{\|v_k(\beta_k)\|_2}{\beta_k}\right] \\
    &\geq \gamma\left[\frac{\|v_k(\beta_k)\|_2}{\beta_k}\right] \min \left\{\frac{1}{1+ \|J_k^T J_k\|_2}\left[\frac{\|v_k(\beta_k)\|_2}{\beta_k}\right], \kappa_v \alpha_k \delta_k \right\},
\end{align*}
so that the first inequality again holds, and completes the proof for this case.

The second inequality  follows from the first inequality and~\eqref{eq:bound-sufficient-decrease-sub}.
\end{proof}

Combining the previous result with Lemma~\ref{lem:vk-monotonicity} gives new lower bounds.
\begin{lemma}\label{cor:cauchy-decrease}
For $\kappa_1 \in (0,1]$ in Lemma~\ref{lem:cauchy-decrease-projection}, the Cauchy point $v_k^c \equiv v_k(\beta_k)$ yields
\begin{subequations}\label{eq:cauchy-decrease-projection-lemma}
\begin{align}
m_k(0) - m_k(v_k^c) 
&\geq \kappa_1 \left[\frac{\|v_k(\beta_k)\|_2}{\beta_k}\right]^2 \min \left\{\frac{1}{1+ \|J_k^T J_k\|_2}, \kappa_v \alpha_k \right\} \label{eq:cauchy-decrease-projection-corollary-1} \\
&\geq \kappa_1 \|v_k(1)\|_2^2 \min \left\{\frac{1}{1+ \|J_k^T J_k\|_2}, \kappa_v \alpha_k \right\} \label{eq:cauchy-decrease-projection-corollary-2}
\end{align}
\end{subequations}
and
\begin{equation}\label{diff-m-no-assumptions}
\|c_k\|_2 - \|c_k + J_k v_k^c\|_2 
\geq \tfrac{\kappa_1}{\kappa_c}\|v_k(1)\|_2^2 \min \left\{\frac{1}{1+ \|J_k^T J_k\|_2}, \kappa_v \alpha_k \right\}.
\end{equation}
\end{lemma}
\begin{proof}
Inequality~\eqref{eq:cauchy-decrease-projection-corollary-1} follows from Lemma~\ref{lem:cauchy-decrease-projection}, $v_k^c = v_k(\beta_k)$, and~\eqref{eq:lower-bound-ratio} with $\beta = \beta_k$. Inequality~\eqref{eq:cauchy-decrease-projection-corollary-2} follows from~\eqref{eq:nonincraesing} since $\beta_k \leq 1$.

It follows from \eqref{eq:cauchy-decrease-projection-corollary-1} that  $\|c_k+J_kv_k^c\|_2 \leq \|c_k\|_2$.  If $\|c_k\|_2 = 0$, then \eqref{diff-m-no-assumptions} follows trivially.  Otherwise, it follows from $\|c_k+J_kv_k^c\|_2 \leq \|c_k\|_2$ that
\begin{equation}~\label{eq:feasibility-reduction-inequality-new}
\begin{aligned}
\|c_k\|_2^2 - \|c_k + J_k v_k^c\|_2^2 &= (\|c_k\|_2 + \|c_k + J_k v_k^c\|_2)(\|c_k\|_2 - \|c_k + J_k v_k^c\|_2) \\
&\leq 2\|c_k\|_2 (\|c_k\|_2 - \|c_k + J_k v_k^c\|_2).
\end{aligned}
\end{equation}
Combining~\eqref{eq:feasibility-reduction-inequality-new} and \eqref{eq:cauchy-decrease-projection-lemma} we have
\begin{align*}
2\|c_k\|_2(\|c_k\|_2 - \|c_k + J_k v_k^c\|_2) 
&\geq \|c_k\|_2^2 - \|c_k + J_k v_k^c\|_2^2 
= 2(m_k(0) - m_k(v_k^c)) \\ 
&\geq 2\kappa_1 \|v_k(1)\|_2^2 \min \left\{\frac{1}{1+ \|J_k^T J_k\|_2}, \kappa_v \alpha_k\right\}.
\end{align*}
Diving both sides by $2\|c_k\|_2$ and using~\eqref{eq:consequence-of-basic-assumption} gives~\eqref{diff-m-no-assumptions}.
\end{proof}

Our next lemma relates the computation of $v_k$ to the measure $\delta_k$. We suspect the first result is well-known in the literature but we could not find a suitable reference.

\begin{lemma}\label{lem:opt-equivalence}
The following results hold.
\begin{itemize}
\item[(i)]
If $\|v_k(1)\|_2 = 0$, then $\delta_k = 0$.
\item[(ii)] $\|v_k\|_2 = 0$ if and only if $\delta_k = 0$.
\item[(iii)] If $\delta_k = 0$, then $x_k$ is a first-order KKT point for problem~\eqref{prob:isp}.
\end{itemize}
\end{lemma}
\begin{proof}
To prove part (i), we suppose that $\|v_k(1)\|_2 = 0$. Note that $0 = \|v_k(1)\|_2 = \|\text{Proj}_{ \Omega}(x_k-J_k^T c_k) - x_k\|_2$ implies that $\text{Proj}_{ \Omega}(x_k-J_k^T c_k) = x_k$. Using this fact, we can apply the  projection theorem~\cite[Proposition 1.1.9]{bertsekas2009convex} to obtain
\begin{equation*}
(-J_k^Tc_k)^T (z-x_k) = (x_k - J_k^Tc_k - x_k)^T (z-x_k) \leq 0 \ \text{for all $z\in\Omega$,}
\end{equation*}
which is equivalent to $-J_k^Tc_k \in N_{\Omega} (x_k)$. It now follows from  Definition~\ref{def:tangent-cone} that
\begin{equation}~\label{eq:duality-normal-tangent}
(-J_k^Tc_k)^T v \leq 0 \ \text{for all $v\in T_{\Omega}(x_k)$.}
\end{equation} 
Using~\eqref{eq:duality-normal-tangent} and nonnegativity of norms, we find that
\begin{equation*}\label{eq:objective-lb}
\tfrac{1}{2} \| v + J_k^T c_k \|_2^2 = \tfrac{1}{2} \left( \|v\|_2^2 + 2v^T J_k^T c_k + \|J_k^T c_k\|_2^2 \right) \geq \tfrac{1}{2}\|J_k^T c_k\|_2^2 \ \text{for all $v\in T_{\Omega}(x_k)$.}
\end{equation*}
It follows from this inequality and 
$\frac{1}{2} \| v + J_k^T c_k \|_2^2$ being strongly convex in $v$ that
$$
0 = \mathop{\arg \min}\limits_{v \in T_{\Omega}(x_k)} \tfrac{1}{2} \|v + J_k^Tc_k\|_2^2 = \mathop{\arg \min}\limits_{v \in T_{\Omega}(x_k)} \|v + J_k^Tc_k\|_2
= \Proj_{T_{\Omega}(x_k)}(-J_k^Tc_k) = \nabla_\Omega(\psi(x_k)).
$$
It now follows from~\eqref{def:deltak} that $\delta_k = 0$, which completes the proof of part (i).

To prove part (ii), we first observe from Algorithm~\ref{alg:main} that if $\delta_k = 0$ then $v_k = 0$.  Thus, it remains to prove that if $v_k = 0$, then $\delta_k = 0$. To do this, let us assume that $v_k = 0$. It follows from the third condition in~\eqref{eq:vk-condition} and Lemma~\ref{cor:cauchy-decrease} that
\begin{equation*}
0 = m_k(0)-m_k(v_k) \geq m_k(0)-m_k(v_k^c) \geq \kappa_1 \|v_k(1)\|_2^2 \min \left\{\frac{1}{1+ \|J_k^T J_k\|_2}, \kappa_v \alpha_k \right\}.
\end{equation*}
Since $\kappa_1$, $\kappa_v$, and $\alpha_k$ are strictly positive, it follows that $\|v_k(1)\|_2 = 0$.  We can combine this result with part (i) to conclude that 
$\delta_k = 0$, which completes the proof.

The proof of part (iii) is provided in~\cite[Lemma 3.1(c)]{calamai1987projected}.
\end{proof}

\subsection{Subproblem (\ref{subprob:stationarity})}
With respect to subproblem~\eqref{subprob:stationarity}, we recall that $u = 0$ is feasible, the constraints are linear (meaning that the feasible region is convex and that a constraint qualification holds), and the objective function is strongly convex. Therefore, the unique solution $u_k$ to subproblem~\eqref{subprob:stationarity} satisfies, for some  $g_{r,k} \in \partial r(x_k + v_k + u_k)$, $y_k \in \R{m}$, and $z_k \in \R{n}$, the following conditions: 
\begin{subequations}\label{kkt:u}
\begin{align}
g_k + \tfrac{1}{\alpha_k}u_k + \tfrac{1}{\alpha_k}v_k + g_{r,k} + J_k^T y_k + z_k &= 0,\label{kkt:stationary} \\
J_ku_k &= 0, \label{kkt:nullspace}\ \ \text{and}\\
\|\min\{x_k + v_k + u_k, -z_k\}\|_2 &= 0, \label{kkt:complementarity}
\end{align}
\end{subequations}
where the minimum of two vectors is taken componentwise. These conditions characterize $u_k$ and will play a critical role in the analysis of Section~\ref{sec:analysis}.  In particular, they allow us to establish the following bound on the size of the trial step.

\begin{lemma}\label{lem:snorm-bound}
The trial step $s_k$ satisfies $\|s_k\|_2 \geq \|\min\{x_k,-z_k\}\|_2$.
\end{lemma}
\begin{proof}
It follows from $s_k = v_k + u_k$ and \eqref{kkt:u} that 
\begin{equation}\label{eq:stationarity-old}
-\tfrac{1}{\alpha_k}s_k = g_k + g_{r,k} + J_k^T y_k +z_k \ \ \text{and} \ \
\|\min\{x_k + s_k, -z_k\}\|_2 = 0.
\end{equation}
The latter equality and min-inequalities give, for each $i\in\{1,2,\dots, n\}$, that 
\begin{align*}
0 &= \min\{[x_k + s_k]_i,-[z_k]_i\}
\geq \min\{[x_k]_i,-[z_k]_i\} + \min\{[s_k]_i, 0\}.
\end{align*}
Combining this inequality with $\min\{[x_k]_i,-[z_k]_i\} \geq 0$ gives
$0 \leq \min\{[x_k]_i,-[z_k]_i\} \leq -\min\{[s_k]_i,0\}$. 
It follows from this inequality that
\begin{align*}
\|\min\{x_k,-z_k\}\|_2^2
&= \sum_{i=1}^n |\min\{[x_k]_i,-[z_k]_i\}|^2 \\
&\leq \sum_{i=1}^n |\min\{[s_k]_i,0\}|^2
\leq \sum_{i=1}^n |[s_k]_i|^2
= \|s_k\|_2^2.
\end{align*}
Taking the square-root of both sides of this inequality completes the proof.
\end{proof}

\section{Analysis}\label{sec:analysis} 
In this section, we present a complete convergence analysis for Algorithm~\ref{alg:main} in both the finite termination case and infinite iteration case.

\subsection{Finite termination}
Our first result shows that the solutions to our subproblems that define the trial step are both zero precisely when the trial step is zero.
\begin{lemma}\label{lem:sk-vk-uk-0}
$s_k = 0$ if and only if $v_k = u_k = 0$.
\end{lemma}
\begin{proof}
Since $s_k = v_k + u_k$, it follows that if $v_k = u_k = 0$, then $s_k = 0$.  Thus, it remains to prove that if $s_k = 0$, then $v_k = u_k = 0$.  For a proof by contradiction, suppose that $s_k = 0$ and $v_k \neq 0 $. It follows from Lemma~\ref{lem:opt-equivalence}(i)(ii) that $v_k(1) \neq 0$, so that  Lemma~\ref{cor:cauchy-decrease} gives  $v_k^c \neq 0$. We may now combine this result with~\eqref{eq:ip-bound} to obtain
\begin{equation*}
c_k^T J_kv_k^c 
= (J_k^T c_k)^T v_k^c 
= \nabla m_k(0)^T v_k^c
\leq -\|v_k^c\|_2^2/\beta_k < 0,
\end{equation*}
which implies that $J_kv_k^c \neq 0$, i.e., that $v_k^c$ is not in the nullspace of $J_k$. At the same time, we know from~\eqref{kkt:nullspace} that $u_k$ is in the nullspace of $J_k$. The previous two statements cannot both be true since $s_k = v_k + u_k = 0$ implies that $v_k = - u_k$, which is a contradiction.  Therefore, we must conclude that $v_k = 0$. Combining this result with $s_k = v_k + u_k = 0$ shows that $u_k = 0$, and completes the proof.
\end{proof}

We can now state our finite termination results for Algorithm~\ref{alg:main}.
\begin{theorem}
The following finite termination results hold for Algorithm~\ref{alg:main}. 
\begin{itemize}
\item[(i)] If Algorithm~\ref{alg:main} terminates at Line~\ref{line:ifs}, then $x_k$ is an infeasible stationary point, i.e., $x_k$ is a first-order KKT point for problem~\eqref{prob:isp} and $\|c_k\|_2 \neq 0$.
\item[(ii)] If Algorithm~\ref{alg:main} terminates at Line~\ref{line:first.order.kkt}, then $x_k$ is a first-order KKT point for problem~\eqref{prob:general}.
\end{itemize}
\end{theorem}
\begin{proof}
We first prove part (i). If Algorithm~\ref{alg:main} terminates at Line~\ref{line:ifs}, then it follows from Lines~\ref{line:deltak} and~\ref{line:ifs.start} that $\delta_k = 0$ and $\|c_k\|_2 \neq 0$.  It now follows from $\delta_k = 0$ and Lemma~\ref{lem:opt-equivalence}(iii) that $x_k$ is a first-order KKT point for problem~\eqref{prob:isp}, as claimed.

For part (ii), we know that if Algorithm~\ref{alg:main} terminates in Line~\ref{line:first.order.kkt} then $s_k = 0$, which from Lemma~\ref{lem:sk-vk-uk-0} implies that $u_k=v_k=0$, and then Lemma~\ref{lem:opt-equivalence}(ii) implies that $\delta_k = 0$. Since termination did not occur in Line~\ref{line:ifs} of Algorithm~\ref{alg:main}, we know that $\|c_k\|_2 = 0$.  It follows from $v_k=u_k=0$ and~\eqref{kkt:u} that there exists $g_{r,k} \in \partial r(x_k)$, $y_k \in \R{m}$, and $z_k \in \R{n}$ satisfying $g_k + g_{r,k} + J_k^T y_k + z_k = 0$ and $\|\min\{x_k,-z_k\}\|_2 = 0$. These equations and $\|c_k\|_2 = 0$ show that $x_k$ is a first-order KKT point for~\eqref{prob:general}. 
\end{proof}

\subsection{Infinite iterations}
We now consider the scenario where finite termination does not occur, meaning that Algorithm~\ref{alg:main} performs an infinite number of iterations. 

\subsubsection{Analysis under no constraint qualification}\label{sec:no-cq}

In this section, we analyze properties of the iterate sequence $\{x_k\}$ generated by Algorithm~\ref{alg:main} when no constraint qualification is assumed to hold. The key metric we consider is
\begin{equation}\label{def:chibark}
\chibar_k := \max\left\{\|g_k + g_{r,k} + J_k^T y_k + z_k\|_2,\|v_k(1)\|_2, \|\max\{x_k,-z_k\}\|_2 \right\},
\end{equation}
where $g_{r,k} \in \R{n}$, $y_k \in \R{m}$, and $z_k \in \R{n}$ are defined as those quantities satisfying~\eqref{kkt:u}. The first quantity in the max is a measure of stationarity for problem~\eqref{prob:general}, the second quantity is a stationarity measure for problem~\eqref{prob:isp}, and the third quantity measures feasibility with respect to $x_k\in\Omega$, the sign of the Lagrange multiplier estimate $z_k$, and complementarity.  In particular, we emphasize that $\|v_k(1)\|_2$ is used here in place of $\|c_k\|_2$ since a constraint qualification is not assumed to hold in this section, meaning that it is possible that the iterates do not converge toward feasibility.

Our first result gives a uniform upper bound on the sequence $\{\delta_k\}$ defined in~\eqref{def:deltak}.

\begin{lemma}\label{lem:vk(1)-bound}
For all iterations $k \in \N{}$, we have that 
\begin{equation}~\label{eq:deltak-ub}
\delta_k \equiv \|\nabla_{\Omega} \psi(x_k)\|_2 \leq 2\kappaJ \|c_k\|_2 \leq 2\kappa_J \kappa_c.
\end{equation}
\end{lemma}
\begin{proof}
Recall that $\nabla_{\Omega} \psi(x_k) = \mathop{\arg\min} \{\|v + J_k^Tc_k\|_2 : v \in T_{\Omega}(x_k)\}$. It follows from this fact, the triangle inequality, and $0 \in T_{\Omega}(x_k)$ that
\begin{equation*}
\|\nabla_{\Omega} \psi(x_k)\|_2 - \|J_k^Tc_k\|_2 \leq \|\nabla_{\Omega} \psi(x_k) + J_k^Tc_k\|_2 \leq \|J_k^Tc_k\|_2.
\end{equation*}
It follows from this inequality, how $\delta_k$ is defined in \eqref{def:deltak}, and Assumption~\ref{ass:basic} that $\delta_k \equiv \|\nabla_{\Omega} \psi(x_k)\|_2 \leq 2\|J_k^Tc_k\|_2  \leq 2\kappa_J \|c_k\|_2 \leq 2\kappa_J \kappa_c$, which completes the proof.
\end{proof}

We can now prove an upper bound on $A_k$ that is defined for $\tauktrial$.

\begin{lemma}\label{lem:denominator}
For all $k \in \N{}$, we have that
\begin{equation*}
g_k^T s_k + \tfrac{1}{2\alpha_k}\|s_k\|_2^2 + r(x_k+s_k) - r_k \leq 2(\kappag + \kappa_{\partial r}) \kappav \kappaJ \alpha_k \|c_k\|_2 + 2\kappav^2 \kappaJ^2 \kappac \alpha_k \|c_k\|_2. 
\end{equation*}
\end{lemma}
\begin{proof}
By convexity of $r$, we know that 
\begin{equation}~\label{eq:subgradient.inequality}
r(x_k + v_k) - r_k \leq (g_{r,k}^v)^T v_k \ \text{for all} \ g_{r,k}^v \in \partial r(x_k + v_k).
\end{equation}
It now follows that
\begin{align*}
&g_k^T s_k + \tfrac{1}{2\alpha_k}\|s_k\|_2^2 + r(x_k+s_k) - r_k \\
\overset{(i)}{\leq} \ & g_k^T v_k + \tfrac{1}{2\alpha_k}\|v_k\|_2^2 + r(x_k+v_k) - r_k \\
\overset{(ii)}{\leq} \ & g_k^T v_k + \tfrac{1}{2\alpha_k}\|v_k\|_2^2 +(g_{r,k}^v)^T v_k \\
\overset{(iii)}{\leq} \ & (\|g_k\|_2 + \|g_{r,k}^v\|_2) \|v_k\|_2 + \tfrac{1}{2\alpha_k} \|v_k\|_2^2 \\
\overset{(iv)}{\leq} \ & (\|g_k\|_2 + \|g_{r,k}^v\|_2) \kappa_v \alpha_k \delta_k + \tfrac{1}{2\alpha_k} \kappa_v^2 \alpha_k^2 \delta_k^2 \\
\overset{(v)}{=} \ & (\|g_k\|_2 + \|g_{r,k}^v\|_2) \kappa_v \alpha_k \delta_k + \thalf\kappa_v^2 \alpha_k \delta_k^2 \\
\overset{(vi)}{\leq} \ & (\|g_k\|_2 + \|g_{r,k}^v\|_2) 2 \kappav \alpha_k \kappaJ \|c_k\|_2 + 2\kappav^2 \alpha_k \kappaJ^2 \kappac \|c_k\|_2 \\
\overset{(vii)}{\leq} \ & (\kappag + \kappa_{\partial r}) 2\kappav \kappaJ \alpha_k \|c_k\|_2 + 2\kappav^2 \kappaJ^2 \kappac \alpha_k \|c_k\|_2,
\end{align*}
where (i) follows from substituting $s_k = v_k + u_k$ and using the fact that $u_k = 0$ is a feasible solution to the tangential subproblem~\eqref{subprob:stationarity}, (ii) follows from~\eqref{eq:subgradient.inequality}, (iii) follows from the Cauchy-Schwartz inequality, (iv) follows from $\|v_k\|_2 \leq \kappa_v \alpha_k \delta_k$ in~\eqref{eq:vk-condition}, (v) follows from canceling an $\alpha_k$ from the second term, (vi) follows from Lemma~\ref{lem:vk(1)-bound} and~\eqref{eq:consequence-of-basic-assumption}, and (vii) follows from~\eqref{eq:consequence-of-basic-assumption}. This completes the proof. 
\end{proof}

The first part of the next lemma establishes that the merit parameter never needs to be decreased for any iteration $k \in \N{}$ such that $v_k(1)=0$. On the other hand, for all $k \in \N{}$ satisfying $v_k(1) \neq 0$, the second part of the lemma provides a lower bound on how small the previous merit parameter $\tau_{k-1}$ could have been when decreased.

\begin{lemma}\label{lem:tau-vk(1)}
The following merit parameter update results hold.
\begin{itemize}
\item[(i)] For each $k\in\N{}\setminus\{0\}$, if $v_k(1) = 0$, then $\tauktrial = \infty$ and $\tau_k \gets \tau_{k-1}$.
\item[(ii)] There exists a constant $\epsilon_{\tau} > 0$ such that, for all $k\in\N{}$ satisfying $\|v_k(1)\|_2 \neq 0$ and $\tau_k < \tau_{k-1}$, it holds that
$\tau_{k-1} \geq \epsilon_{\tau}\|v_k(1)\|_2^2$.
\end{itemize}
\end{lemma}
\begin{proof}
We first prove part~(i).  To this end, first observe that $v_k(1) = 0$ and Lemma~\ref{lem:opt-equivalence}(i) imply that $\delta_k=0$, and therefore $v_k = 0$ holds as a consequence of Lemma~\ref{lem:opt-equivalence}(ii). Next, since $u = 0$ is feasible for subproblem~\eqref{subprob:stationarity} we know that
$$
g_k^Tu_k + \tfrac{1}{2\alpha_k}\|u_k\|_2^2 + \tfrac{1}{\alpha_k}v_k^T u_k + r(x_k + v_k + u_k) \leq r(x_k + v_k),
$$
which may be combined with $v_k = 0$ to obtain
$g_k^Ts_k + \tfrac{1}{2\alpha_k}\|s_k\|_2^2 + r(x_k + s_k) \leq r(x_k)$.
This inequality and the definition of $\tauktrial$ gives $\tauktrial = \infty$, so that $\tau_k \gets \tau_{k-1}$.

Next, we prove part~(ii).  It follows from the merit parameter update rule~\eqref{eq:tau-update}, $J_k u_k = 0$ (see~\eqref{kkt:nullspace}), the third condition in~\eqref{eq:vk-condition}, \eqref{diff-m-no-assumptions}, \eqref{eq:consequence-of-basic-assumption}, Lemma~\ref{lem:denominator}, and monotonicity of the proximal parameter sequence $\{\alpha_k\}$ that if $\tau_k < \tau_{k-1}$, then 
\begin{align*}
\tau_{k-1} 
&> \frac{(1-\sigma_c)(\|c_k\|_2 - \|c_k + J_k v_k\|_2)}{g_k^T s_k + \tfrac{1}{2\alpha_k}\|s_k\|_2^2 + r(x_k+s_k) - r_k} \\
&\geq 
\frac{(1-\sigma_c)(\|c_k\|_2 - \|c_k + J_k v_k^c\|_2)}{g_k^T s_k + \tfrac{1}{2\alpha_k}\|s_k\|_2^2 + r(x_k+s_k) - r_k} \\
&\geq  \frac{(1-\sigma_c)\frac{\kappa_1}{\kappa_c} \|v_k(1)\|_2^2 \min \left\{\frac{1}{1+ \|J_k^T J_k\|_2}, \kappa_v \alpha_k\right\}}{2(\kappag + \kappa_{\partial r}) \kappav \kappaJ \alpha_k \|c_k\|_2 + 2\kappav^2 \kappaJ^2 \kappac \alpha_k \|c_k\|_2} \\
&\geq  \frac{(1-\sigma_c)\kappa_1 \|v_k(1)\|_2^2 \min \left\{\frac{1}{1+ \|J_k^T J_k\|_2}, \kappa_v \alpha_k\right\}}{2(\kappag + \kappa_{\partial r}) \kappav \kappaJ \kappac^2 \alpha_k + 2\kappav^2 \kappaJ^2 \kappac^2 \alpha_k} 
\geq  \epsilon_{\tau} \|v_k(1)\|_2^2,
\end{align*}
where
$
\epsilon_\tau
:= \tfrac{(1-\sigma_c)\kappa_1 \min \left\{\frac{1}{(1+ \kappaJ^2)\alpha_0}, \kappa_v \right\}}{2(\kappag + \kappa_{\partial r}) \kappav \kappaJ \kappac^2 + 2\kappav^2 \kappaJ^2 \kappac^2} > 0
$, thus completing the proof.
\end{proof}

Next, under the assumption that the merit parameter sequence stays bounded away from zero, we give a positive lower bound on $\{\alpha_k\}$.
\begin{lemma}\label{lem:alpha-bounded}
Assume that there exists $\taumin > 0$ such that $\tau_k \geq \taumin$ for all $k\in\N{}$. If $\alpha_k \leq \tfrac{\taumin}{2(\taumin \lipg + \lipJ)}$, then $k\in\Scal$.  Thus, for all $k\in\N{}$,
\begin{equation}\label{eq:lb-alpha-dependence-on-tau}
\alpha_k 
\geq \alphamin 
:= 
\min\{\alpha_0, \tfrac{\xi\taumin}{2(\taumin \lipg + \lipJ)}  \} > 0
\end{equation}
and a bound on the number of unsuccessful iterations is given by
\begin{equation}~\label{bd-unsuccessful}
|\{k \in \N{} : x_k \notin \Scal\}| 
\leq \max\left(0, 
\left\lceil
\frac{\log\Big(\frac{\taumin}{2\alpha_0(\taumin \lipg + \lipJ)}\Big)}{\log(\xi)}
\right\rceil
\right).
\end{equation}
\end{lemma}
\begin{proof}
It follows from~\eqref{eq.Lipschitz} and the merit parameter update rule~\eqref{eq:tau-update} that
\begin{equation}\label{eq:merit-function-decrease}
\begin{aligned}
&\Phi_{\tau_k}(x_k+s_k) - \Phi_{\tau_k}(x_k) \\
=\ & \tau_k\big( f(x_k+s_k) + r(x_k+s_k)\big) + \|c(x_k+s_k)\|_2 - \tau_k\big( f_k + r_k\big) - \|c_k\|_2.\\
\leq \ & \tau_k g_k^T s_k + \tau_k\big(r(x_k+s_k) - r_k \big) + \|c_k+J_ks_k\|_2 - \|c_k\|_2 + \thalf(\tau_k \lipg+\lipJ)\|s_k\|_2^2 \\
\leq \ & -\tfrac{\tau_k}{4\alpha_k}\|s_k\|_2^2 - \sigma_c (\|c_k\|_2 - \|c_k + J_k s_k\|_2) + \thalf(-\tfrac{\tau_k}{2\alpha_k} + \tau_k \lipg+\lipJ)\|s_k\|_2^2.
\end{aligned}
\end{equation}
Suppose that $k \in \N{}$ satisfies $\alpha_k \leq \tfrac{\taumin}{2(\taumin \lipg + \lipJ)}$. It follows from the fact that $\tfrac{\tau}{2(\tau L_g + L_J)}$ is a monotonically
increasing function on the nonnegative real line as a function of $\tau$ that $\alpha_k \leq \tfrac{\taumin}{2(\taumin \lipg + \lipJ)} \leq \tfrac{\tau_k}{2(\tau_k \lipg + \lipJ)}$, which after rearrangement shows that $-\tfrac{\tau_k}{2\alpha_k} + \tau_k L_g + L_J \leq 0$. The previous inequality, $\|s_k\|_2 \neq 0$ (since  finite termination does not occur), \eqref{diff-m-no-assumptions}, $\|c_k+J_kv_k\|_2 \leq \|c_k+J_k v_k^c\|_2$, $J_k u_k = 0$, and $\eta_\Phi \in (0,1)$ give
\begin{equation*}
(1-\eta_\Phi)(\tfrac{\tau_k}{4\alpha_k}\|s_k\|_2^2 + \sigma_c (\|c_k\|_2 - \|c_k + J_k s_k\|_2)) > 0 \geq \thalf(-\tfrac{\tau_k}{2\alpha_k} + \tau_k \lipg+\lipJ)\|s_k\|_2^2.
\end{equation*}
Combining this inequality with~\eqref{eq:merit-function-decrease} shows that $k \in \Scal$, as claimed. This result and the update strategy for the proximal parameter $\alpha_k$ ensures that the bound in~\eqref{eq:lb-alpha-dependence-on-tau} holds. Finally, the first result we proved in this lemma and the update strategy for $\{\alpha_k\}$ shows that the maximum number of unsuccessful iterations is the smallest nonnegative integer $n_u$ such that $\xi^{n_u}\alpha_0 \leq \tfrac{\taumin}{2(\taumin \lipg + \lipJ)}$, which gives the final result.
\end{proof}

It will be convenient for our analysis to define the shifted merit function
\begin{equation}~\label{eq:shifted-merit-function}
\Phibar_{\tau}(x) := \tau\big(f(x) - \finf + r(x) \big) + \|c(x)\|_2,
\end{equation}
where $\finf$ is defined in~\eqref{eq:consequence-of-basic-assumption}. We stress that the (typically) unknown value $\finf$ is never used in the algorithm statement or its implementation, only in our analysis.

\begin{lemma}\label{lem:taubar-props}
The following properties hold for the shifted merit function.
\begin{itemize}
\item[(i)] For all $\{x,y\} \subset\R{n}$ and $\tau\in\R{}_{>0}$, it holds that $\Phibar_\tau(x) - \Phibar_\tau(y) = \Phi_\tau(x) - \Phi_\tau(y)$.
\item[(ii)] For all $x\in\R{n}$ and $0 < \tau_2 \leq \tau_1$, it holds that $\Phibar_{\tau_2}(x) \leq \Phibar_{\tau_1}(x)$. 
\item[(iii)] The sequence $\{\Phibar_{\tau_k}(x_k)\}$ is monotonically decreasing.
\end{itemize}
\end{lemma}
\begin{proof}
See~\cite[Lemma 3.14]{dai2024proximal} for a proof.
\end{proof}

We can now state our main convergence result for this section.

\begin{theorem}\label{thm:no-licq}
Let Assumption~\ref{ass:basic} hold. One of the following two cases occurs.
\begin{itemize}
\item[(i)] There exists $\taumin > 0$ such that $\tau_k \geq \taumin$ for all $k\in\N{}$.  In this case, the following hold: (a) $\alpha_k \geq \alphamin 
:= 
\min\{\alpha_0, \tfrac{\xi\taumin}{2(\taumin \lipg + \lipJ)}\}$ for all $k\in\N{}$; (b) If $\{k_1, k_2\} \subset \N{}$ are two iterations with $k_1 < k_2$ such that $k \in \Scal$ and $\chibar_k > \epsilon$ for all iterations $k_1 \leq k < k_2$, then it follows that 
\begin{equation}\label{eq:diff-of-its}
k_2 - k_1 \leq \left\lfloor \frac{\tau_0 (f(x_0) + r(x_0) - f_{\inf}) + \|c(x_0)\|_2}{\kappabarPhi \epsilon^2} \right\rfloor 
\end{equation}
with $\kappabar_{\Phi} = \eta_\Phi \min\left\{\frac{\taumin \alphamin}{8}, \frac{\taumin}{8\alpha_0}, \frac{\sigma_c\kappa_1}{\kappa_c}\min \{\frac{1}{1+\kappaJ^2}, \kappa_v \alpha_{\min}\} \right\}$; and (c) for any given $\epsilon > 0$, the maximum number of iterations before $\chibar_k \leq \epsilon$ is
\begin{equation*}
\!\!\!\!\!\!\!\!\!\!\!\!\left(
\!\max\left\{0,
\left\lceil
\frac{\log\Big(\frac{\taumin}{2\alpha_0(\taumin \lipg + \lipJ)}\Big)}{\log(\xi)}
\right\rceil
\right\}
+ 1 \!\right)
\!\!\left\lfloor \frac{\tau_0\big( f(x_0) - \finf + r(x_0) \big) + \|c(x_0)\|_2}{\kappabarPhi\epsilon^2} \right\rfloor\!.
\end{equation*}
\item[(ii)] The merit parameter values converge to zero, i.e., $\lim_{k\to\infty} \tau_k = 0$.  In this case, there exists a subsequence $\Kcal\subseteq\N{}$ such that $\lim_{k\in\Kcal} \|v_k(1)\|_2 = 0$. 
\end{itemize}
\end{theorem}
\begin{proof}
To prove part (i), let us assume there exists $\taumin > 0$ such that $\tau_k \geq \taumin$ for all $k \in \N{}$. Using this fact, Lemma~\ref{lem:alpha-bounded} ensures that both~\eqref{eq:lb-alpha-dependence-on-tau} and~\eqref{bd-unsuccessful} hold. Since~\eqref{eq:lb-alpha-dependence-on-tau} holds, part~(i)(a) is proved. To prove part~(i)(b), let $\{k_1,k_2\}$ be as in the statement of the theorem. Then, for all $k \in \Scal$ and $k_1 \leq k < k_2$, it follows from Lemma~\ref{lem:taubar-props}(i)--(ii), $k\in\Scal$, \eqref{eq:consequence-of-basic-assumption}, $J_k u_k = 0$, \eqref{diff-m-no-assumptions}, and Lemma~\ref{lem:alpha-bounded} that
\begin{equation}~\label{eq:merit-function-reduction}
\begin{aligned}
& \Phibar_{\tau_k}(x_k) - \Phibar_{\tau_{k+1}}(x_{k+1}) 
\geq  \Phibar_{\tau_k}(x_k) - \Phibar_{\tau_k}(x_{k+1}) 
=  \Phi_{\tau_k}(x_k) - \Phi_{\tau_k}(x_{k+1}) \\
\geq \ & \eta_\Phi \left( \tfrac{\tau_k}{4\alpha_k} \|s_k\|_2^2 + \sigma_c ( \|c_k\|_2 - \|c_k + J_k s_k\|_2) \right) \\
\geq \ & \eta_\Phi \left[ \tfrac{\tau_k \alpha_k}{4}\left(\tfrac{\|s_k\|_2}{\alpha_k}\right)^2 + 
\tfrac{\sigma_c\kappa_1}{\kappa_c}\|v_k(1)\|_2^2 \min \left\{\tfrac{1}{1+ \|J_k^T J_k\|_2}, \kappa_v \alpha_k \right\}
\right] \\
= \ & \eta_\Phi \left[ \tfrac{\tau_k \alpha_k}{8}\left(\tfrac{\|s_k\|_2}{\alpha_k}\right)^2 + \tfrac{\tau_k \|s_k\|_2^2}{8\alpha_k} + 
\tfrac{\sigma_c\kappa_1}{\kappa_c}\|v_k(1)\|_2^2 \min \left\{\tfrac{1}{1+ \|J_k^T J_k\|_2}, \kappa_v \alpha_k \right\}
\right].
\end{aligned}
\end{equation}
Lemma~\ref{lem:snorm-bound},  \eqref{eq:merit-function-reduction}, 
\eqref{kkt:u}, \eqref{eq:lb-alpha-dependence-on-tau}, and $\tau_k \geq\taumin$ and $\alpha_k \leq \alpha_0$ for all $k\in\N{}$ give
\begin{align*}
& \Phibar_{\tau_k}(x_k) - \Phibar_{\tau_{k+1}}(x_{k+1}) \\
\geq \ & \eta_\Phi \left[ \tfrac{\tau_k \alpha_k}{8}\|g_k + g_{r,k} + J_k^T y_k + z_k\|_2^2 + \tfrac{\tau_k}{8\alpha_k}\|\min\{x_k,-z_k\}\|_2^2 \right. \\
& \left. + 
\tfrac{\sigma_c\kappa_1}{\kappa_c}\|v_k(1)\|_2^2 \min \left\{\tfrac{1}{1+ \|J_k^T J_k\|_2}, \kappa_v \alpha_k \right\}
\right] \\
\geq \ & \eta_\Phi \left[ \tfrac{\tau_{\min} \alpha_{\min}}{8}\|g_k + g_{r,k} + J_k^T y_k + z_k\|_2^2 + \tfrac{\tau_{\min}}{8\alpha_0}\|\min\{x_k,-z_k\} \|_2^2 \right. \\
& \left. +
\tfrac{\sigma_c\kappa_1}{\kappa_c}\|v_k(1)\|_2^2 \min \left\{\tfrac{1}{1+ \kappaJ^2}, \kappa_v \alphamin \right\}
\right] \\
\geq \ & \kappabarPhi \chibar_k^2
\end{align*}
where $\kappabarPhi$ is defined in the statement of the current theorem. Using this inequality, Lemma~\ref{lem:taubar-props}(iii), and nonnegativity of $\Phibar_{\tau}$ for all $\tau\in\R{}_{>0}$, we find that 
\begin{align*}
    \Phibar_{\tau_0}(x_0)
    &\geq \Phibar_{\tau_{k_1}}(x_{k_1}) 
    \geq \Phibar_{\tau_{k_1}}(x_{k_1}) - \Phibar_{\tau_{k_2}}(x_{k_2})  \\
    &= \sum_{k = k_1}^{k_2-1} \big(\Phibar_{\tau_k}(x_k) - \Phibar_{\tau_{k+1}}(x_{k+1}) \big) 
    \geq \sum_{k = k_1}^{k_2-1}\kappabarPhi \chibar_k^2, 
\end{align*}
which may be combined with $\chibar_k > \epsilon$ for all $k_1 \leq k \leq k_2$ to conclude that
$\Phibar_{\tau_0}(x_0)
\geq (k_2-k_1)\kappabarPhi\epsilon^2$, from which~\eqref{eq:diff-of-its} follows. The result (i)(c), namely the claimed upper bound on the maximum iterations before $\chibar_k \leq \epsilon$, follows from what we just proved and the fact that maximum number of unsuccessful iterations is bounded as in~\eqref{bd-unsuccessful}.


We prove part (ii) by contradiction. Thus, suppose that there exists $\epsilon \in\R{}_{>0}$ and $\kbar_1\in\N{}$ such that $\|v_k(1)\|_2 \geq \epsilon$ for all $k \geq \kbar_1$. It then follows from Lemma~\ref{lem:tau-vk(1)} that there exists $\taumin\in\R{}_{>0}$ such that $\tau_k \geq \taumin$ for all $k\in\N{}$, which is a contradiction.
\end{proof}

\subsubsection{Analysis under a sequential constraint qualification}\label{sec:sequential-licq}

In this section, we assume that a sequential constraint qualification holds (all results from Section~\ref{sec:no-cq} still hold). To state this assumption, we define the index set of active variables after taking the Cauchy step $v_k^c$ as 
$$
\Acal_k^v := \Acal(x_k+v_k^c) \equiv \{i \in [n]: [x_k + v_k^c]_i = 0 \}.
$$
We can now formally state the assumption we make throughout this section.
\begin{assumption}\label{ass:strong-licq}
The matrix $[J_k^T, I_{\Acal_k^v}^T]^T$ has full row rank and its smallest singular value is uniformly bounded away from zero for all $k\in\N{}$, where $I_{\Acal_k^v}$ denotes the subset of rows of the identity matrix that correspond to the elements in $\Acal_k^v$, i.e., there exists $\sigmamin \in \R{}_{>0}$ such that $\sigmamin([J_k^T,I_{A_k^v}^T]^T) \geq \sigmamin$ for all $k \in \N{}$ with $\sigmamin(A)$ denoting the smallest singular value of a matrix $A$.
\end{assumption}


Under the above assumption, our aim is to prove a worst-case iteration complexity result for Algorithm~\ref{alg:main}.  Our result uses the KKT-residual measure
\begin{equation}\label{eq:chik}
\chi_k := \max\left\{\|g_k + g_{r,k} + J_k^T y_k + z_k\|_2,\|c_k\|_2, \|\min\{x_k,-z_k\}\|_2 \right\}.
\end{equation}
Note that~\eqref{eq:chik}  differs from the definition of $\chibar_k$ in~\eqref{def:chibark} by using the measure $\|c_k\|_2$ instead of $\|v_k(1)\|_2$, which is reasonable because of the constraint qualification.

We begin by establishing a key connection between $\|v_k(\beta_k)\|_2$ and $\|c_k\|_2$. 
\begin{lemma}\label{lem:vk-ck-relation}
For all $k \in \N{}$, it holds that
$\|v_k(\beta_k)\|_2/\beta_k \geq \sigmamin \|c_k\|_2$. 
\end{lemma}
\begin{proof}
Let us define the vector $w_k\in\R{n}$ componentwise as
\begin{equation}~\label{eq:multiplier}
[w_k]_i = 
\begin{cases}
0 & i \in [n]\setminus \Acal_k^v, \\
-[J_k^Tc_k]_i - [v_k(\beta_k)]_i/\beta_k & i \in \Acal_k^v. 
\end{cases}
\end{equation}
We claim that the following holds:
\begin{equation}\label{proj-claim}
\Proj_{ \Omega} (x_k - \beta_k J_k^Tc_k) - x_k = - \beta_k J_k^Tc_k - \beta_k w_k,
\end{equation}
which we verify by considering its coordinates. If $i \in \Acal_k^v$, then~\eqref{eq:beta-condition} and~\eqref{eq:multiplier} give
\begin{equation}~\label{eq:case1}
\begin{aligned}
&[\Proj_{ \Omega} (x_k - \beta_k J_k^Tc_k) - x_k]_i = [v_k(\beta_k)]_i \\
= \  &[-\beta_k J_k^Tc_k]_i - [-\beta_k J_k^Tc_k - v_k(\beta_k)]_i = [-\beta_k J_k^Tc_k]_i - [\beta_k w_k]_i,
\end{aligned}
\end{equation}
so that~\eqref{proj-claim} holds in this case. On the other hand, if $i \in [n]\setminus \Acal_k^v$, then $[\Proj_{ \Omega} (x_k - \beta_k J_k^Tc_k)]_i = [x_k + v_k(\beta_k)]_i  = [x_k + v_k^c]_i > 0$ and $[w_k]_i = 0$. It follows that 
\begin{equation}
0 < [\Proj_{ \Omega} (x_k - \beta_k J_k^Tc_k)]_i = \max\left\{[x_k - \beta_k J_k^Tc_k]_i, 0\right\},
\end{equation}
which implies that $[x_k - \beta_k J_k^Tc_k]_i > 0$.  Combining this with $[w_k]_i = 0$ shows that
\begin{equation}\label{eq:case2}
\begin{aligned}
\ [\Proj_{ \Omega} (x_k - \beta_k J_k^Tc_k) - x_k]_i &= [(x_k - \beta_k J_k^Tc_k) - x_k]_i \\
&= [- \beta_k J_k^Tc_k]_i 
= [- \beta_k J_k^Tc_k - \beta_k w_k]_i
\end{aligned}
\end{equation}
so that~\eqref{proj-claim} again holds for this case. This establishes that~\eqref{proj-claim} holds, as claimed.  It follows from the definition of $v_k(\beta_k)$, \eqref{proj-claim}, and Assumption~\ref{ass:strong-licq} that 
\begin{equation*}
\begin{aligned}
\left\|\frac{v_k(\beta_k)}{\beta_k}\right\|_2 
&= \left\|\frac{\Proj_{ \Omega} (x_k - \beta_k J_k^Tc_k) - x_k}{\beta_k}\right\|_2 
= \left\| \frac{- \beta_k J_k^Tc_k - \beta_k w_k}{\beta_k} \right\|_2 \\
&= \|J_k^Tc_k + w_k\|_2 
= \left\|\begin{bmatrix}
    J_k^T, I_{\Acal_k^v}^T 
\end{bmatrix} 
\begin{bmatrix}
    c_k \\
    [w_k]_{\Acal_k^v}
\end{bmatrix}\right\|_2 \\
&\geq \sigmamin([J_k^T,I_{A_k^v}^T]^T) \left\|\begin{bmatrix}
    c_k \\
    [w_k]_{\Acal_k^v}
\end{bmatrix}\right\|_2 \geq \sigmamin \|c_k\|_2 \ \text{for all $k \in \N{}$,}
\end{aligned} 
\end{equation*}
which completes the proof.
\end{proof}

We now give a bound on the improvement in linearized infeasibility at $x_k$.
\begin{lemma}\label{lem:numerator}
For all $k \in \N{}$, it holds that
\begin{equation*}
\|c_k\|_2 - \|c_k + J_k s_k\|_2
=
\|c_k\|_2 - \|c_k + J_k v_k\|_2 \geq \kappa_1\sigmamin^2 \|c_k\|_2 \min \left\{\tfrac{1}{1+ \|J_k^T J_k\|_2}, \kappa_v \alpha_k\right\}.
\end{equation*}
and
$$
\|c_k\|_2 - \|c_k + J_k s_k\|_2
= \|c_k\|_2 - \|c_k + J_k v_k\|_2
\geq \frac{\kappa_1}{\kappa_c} \sigmamin^2 \|c_k\|_2^2 \min\left\{\tfrac{1}{1+\|J_k^T J_k\|_2}, \kappav\alpha_k\right\}.
$$
\end{lemma}
\begin{proof}
It follows from~\eqref{eq:vk-condition} and Lemma~\ref{lem:cauchy-decrease-projection} that  $\|c_k+J_kv_k\|_2 \leq \|c_k\|_2$.  It follows from this inequality and a  difference-of-squares computation that
\begin{equation}~\label{eq:feasibility-reduction-inequality}
\begin{aligned}
\|c_k\|_2^2 - \|c_k + J_k v_k\|_2^2 &= (\|c_k\|_2 + \|c_k + J_k v_k\|_2)(\|c_k\|_2 - \|c_k + J_k v_k\|_2) \\
&\leq 2\|c_k\|_2 (\|c_k\|_2 - \|c_k + J_k v_k\|_2).
\end{aligned}
\end{equation}
Combining~\eqref{eq:feasibility-reduction-inequality}, the third condition in~\eqref{eq:vk-condition}, Lemma~\ref{cor:cauchy-decrease},  and Lemma~\ref{lem:vk-ck-relation} we have 
\begin{align*}
& 2\|c_k\|_2(\|c_k\|_2 - \|c_k + J_k v_k\|_2) \geq \|c_k\|_2^2 - \|c_k + J_k v_k\|_2^2 
= 2(m_k(0) - m_k(v_k)) \\ 
&\geq 2(m_k(0) - m_k(v_k^c)) 
\geq 2\kappa_1 \left[\tfrac{\|v_k(\beta_k)\|_2}{\beta_k}\right]^2 \min \left\{\tfrac{1}{1+ \|J_k^T J_k\|_2}, \kappa_v \alpha_k\right\} \\ 
&\geq  2\kappa_1 \sigmamin^2 \|c_k\|_2^2 \min \left\{\tfrac{1}{1+ \|J_k^T J_k\|_2}, \kappa_v \alpha_k\right\}.
\end{align*}
The proof of the first inequality follows by  dividing through the previous inequality by $2\|c_k\|_2$ and using the fact that $J_k u_k = 0$ (see~\eqref{kkt:nullspace}). The second inequality follows from the first inequality and the fact that $\|c_k\|_2/\kappac \leq 1$ because of~\eqref{eq:consequence-of-basic-assumption}.
\end{proof}

We now establish that the merit parameter sequence is bounded away from zero.
\begin{lemma}\label{lem:tau-bounded}
For all $k \in \N{}$, it holds that
\begin{align}
\tauktrial &\geq \tau_{\min,trial} := \frac{(1-\sigma_c)\kappa_1\sigmamin^2 \min \left\{\frac{1}{(1+ \kappaJ^2)\alpha_0} , \kappa_v\right\}}{2(\kappag + \kappa_{\partial r}) \kappav \kappaJ + 2\kappav^2 \kappaJ^2 \kappac} > 0 \ \ \text{and} \label{bound-tautrial} \\
\tau_k &\geq \tau_{\min}:=\min\{\tau_0, (1-\epsilon_{\tau})\tau_{\min,trial})\} > 0.\label{bound-tau}
\end{align}
\end{lemma}
\begin{proof}
We first prove~\eqref{bound-tautrial}.  If $A_k \leq 0$ in the definition of $\tauktrial$, then $\tauktrial = \infty$ so that~\eqref{bound-tautrial} trivially holds.  If $A_k > 0$, then it follows from the definition of $\tauktrial$, $s_k = v_k + u_k$, $J_k u_k = 0$ (see~\eqref{kkt:nullspace}), Lemma~\ref{lem:numerator}, Lemma~\ref{lem:denominator}, the fact that $\alpha_k \leq \alpha_0$ for all $k$ by construction of Algorithm~\ref{alg:main}, and~\eqref{eq:consequence-of-basic-assumption} that
\begin{align*}
\tau_{k,\text{trial}} 
&= \frac{(1-\sigma_c)(\|c_k\|_2 - \|c_k + J_k v_k\|_2)}{g_k^T s_k + \tfrac{1}{2\alpha_k}\|s_k\|_2^2 + r(x_k+s_k) - r_k} \\
&\geq \frac{(1-\sigma_c)\kappa_1\sigmamin^2 \|c_k\|_2 \min \left\{\frac{1}{1+ \|J_k^T J_k\|_2}, \kappa_v \alpha_k\right\}}{2(\kappag + \kappa_{\partial r}) \kappav \kappaJ \alpha_k \|c_k\|_2 + 2\kappav^2 \kappaJ^2 \kappac \alpha_k \|c_k\|_2} \\
&= \frac{(1-\sigma_c)\kappa_1\sigmamin^2 \min \left\{\frac{1}{1+ \|J_k^T J_k\|_2}, \kappa_v \alpha_k\right\}}{2(\kappag + \kappa_{\partial r}) \kappav \kappaJ \alpha_k + 2\kappav^2 \kappaJ^2 \kappac \alpha_k} \\
&\geq \frac{(1-\sigma_c)\kappa_1\sigmamin^2 \min \left\{\frac{1}{(1+ \kappaJ^2)\alpha_0} , \kappa_v\right\}}{2(\kappag + \kappa_{\partial r}) \kappav \kappaJ + 2\kappav^2 \kappaJ^2 \kappac},
\end{align*}
which proves~\eqref{bound-tautrial}. The merit parameter update rule~\eqref{eq:tau-update} and~\eqref{bound-tautrial} give~\eqref{bound-tau}.
\end{proof}

We may now state our worst-case complexity result for Algorithm~\ref{alg:main}.
\begin{theorem}\label{thm:complexity}
Suppose that Assumption~\ref{ass:basic} and Assumption~\ref{ass:strong-licq} hold. Let $\epsilon \in \R{}_{>0}$ be given. If $\{k_1, k_2\} \subset \N{}$ are two iterations with $k_1 < k_2$ such that $k \in \Scal$ and $\chi_k > \epsilon$ for all iterations $k_1 \leq k < k_2$, then it follows that 
\begin{equation}\label{eq:diff-of-its-old}
k_2 - k_1 \leq \left\lfloor \frac{\tau_0 (f(x_0) + r(x_0) - f_{\inf}) + \|c(x_0)\|_2}{\kappaPhi \epsilon^2} \right\rfloor 
\end{equation}
with $\kappa_{\Phi} = \eta_\Phi \min\left\{\frac{\taumin \alphamin}{8}, \frac{\taumin}{8\alpha_0}, \frac{\sigma_c\kappa_1}{\kappa_c}\sigmamin^2 \min \{\frac{1}{1+\kappaJ^2}, \kappa_v \alpha_{\min}\}  \right\}$. Moreover, the maximum number of iterations before $\chi_k \leq \epsilon$ for some iteration $k\in\N{}$ is
\begin{equation*}
\left(
\max\left\{0,
\left\lceil
\frac{\log\Big(\frac{\taumin}{2\alpha_0(\taumin \lipg + \lipJ)}\Big)}{\log(\xi)}
\right\rceil
\right\}
 + 1 \right)
\left\lfloor \frac{\tau_0\big( f(x_0) - \finf + r(x_0) \big) + \|c(x_0)\|_2 }{\kappaPhi\epsilon^2} \right\rfloor.
\end{equation*}
\end{theorem}
\begin{proof}
Let $\{k_1,k_2\}$ be as in the statement of the theorem. Then, for all $k \in \Scal$ and $k_1 \leq k < k_2$, it follows from Lemma~\ref{lem:taubar-props}(i)--(ii), $k\in\Scal$, \eqref{eq:consequence-of-basic-assumption}, the second inequality of Lemma~\ref{lem:numerator}, and Lemma~\ref{lem:alpha-bounded} that
\begin{equation}\label{eq:merit-function-reduction-old}
\begin{aligned}
& \Phibar_{\tau_k}(x_k) - \Phibar_{\tau_{k+1}}(x_{k+1}) 
\geq \Phibar_{\tau_k}(x_k) - \Phibar_{\tau_k}(x_{k+1}) 
= \Phi_{\tau_k}(x_k) - \Phi_{\tau_k}(x_{k+1}) \\
\geq \ & \eta_\Phi \left( \tfrac{\tau_k}{4\alpha_k} \|s_k\|_2^2 + \sigma_c ( \|c_k\|_2 - \|c_k + J_k s_k\|_2) \right) \\
\geq \ & \eta_\Phi \left[ \tfrac{\tau_k \alpha_k}{4}\left(\tfrac{\|s_k\|_2}{\alpha_k}\right)^2 + \sigma_c \left( \tfrac{\kappa_1}{\kappa_c} \sigmamin^2 \|c_k\|_2^2 \min\left\{\tfrac{1}{1+\kappaJ^2}, \kappav\alpha_k\right\} \right) \right] \\
= \ & \eta_\Phi \left[ \tfrac{\tau_k \alpha_k}{8}\left(\tfrac{\|s_k\|_2}{\alpha_k}\right)^2 + \tfrac{\tau_k \|s_k\|_2^2}{8\alpha_k} + \sigma_c \left( \tfrac{\kappa_1}{\kappa_c} \sigmamin^2 \|c_k\|_2^2 \min\left\{\tfrac{1}{1+\kappaJ^2}, \kappav\alpha_{\min} \right\} \right) \right].
\end{aligned}
\end{equation}
Lemma~\ref{lem:snorm-bound}, \eqref{eq:merit-function-reduction-old}, \eqref{kkt:u}, \eqref{bound-tau}, \eqref{eq:lb-alpha-dependence-on-tau}, and $\alpha_k \leq \alpha_0$ for all $k \geq 0$ give
\begin{align*}
& \Phibar_{\tau_k}(x_k) - \Phibar_{\tau_{k+1}}(x_{k+1}) \\
\geq \ & \eta_\Phi \left[ \tfrac{\tau_k \alpha_k}{8}\|g_k + g_{r,k} + J_k^T y_k + z_k\|_2^2 + \tfrac{\tau_k}{8\alpha_k}\|\min\{x_k,-z_k\}\|_2^2 \right. \\
& \left. + \sigma_c \left( \tfrac{\kappa_1}{\kappa_c} \sigmamin^2 \|c_k\|_2^2 \min\left\{\tfrac{1}{1+\kappaJ^2}, \kappav\alpha_{\min}\right\} \right) \right] \\
\geq \ & \eta_\Phi \left[ \tfrac{\tau_{\min} \alpha_{\min}}{8}\|g_k + g_{r,k} + J_k^T y_k + z_k\|_2^2 + \tfrac{\tau_{\min}}{8\alpha_0}\|\min\{x_k,-z_k\} \|_2^2 \right. \\
& \left. + \sigma_c \left( \tfrac{\kappa_1}{\kappa_c} \sigmamin^2 \|c_k\|_2^2 \min\left\{\tfrac{1}{1+\kappaJ^2}, \kappav\alpha_{\min}\right\} \right) \right] \\
\geq \ & \kappaPhi \chi_k^2
\end{align*}
where $\kappaPhi$ is defined in the statement of the current theorem. Using this inequality, Lemma~\ref{lem:taubar-props}(iii), and nonnegativity of $\Phibar_{\tau}$ for all $\tau\in\R{}_{>0}$, we find that 
\begin{align*}
    \Phibar_{\tau_0}(x_0)
    &\geq \Phibar_{\tau_{k_1}}(x_{k_1}) 
    \geq \Phibar_{\tau_{k_1}}(x_{k_1}) - \Phibar_{\tau_{k_2}}(x_{k_2})  \\
    &= \sum_{k = k_1}^{k_2-1} \big(\Phibar_{\tau_k}(x_k) - \Phibar_{\tau_{k+1}}(x_{k+1}) \big) 
    \geq \sum_{k = k_1}^{k_2-1}\kappaPhi \chi_k^2, 
\end{align*}
which may be combined with $\chi_k > \epsilon$ for all iterations $k_1 \leq k \leq k_2$ to conclude that
$$
\Phibar_{\tau_0}(x_0)
\geq (k_2-k_1)\kappaPhi\epsilon^2,
$$
from which~\eqref{eq:diff-of-its} follows. The final result in the theorem, namely the claimed upper bound on the maximum iterations before $\chi_k \leq \epsilon$, follows from what we just proved and the fact that maximum number of unsuccessful iterations is bounded as in~\eqref{bd-unsuccessful}.
\end{proof}

\subsubsection{Analysis under a limit-point constraint qualification}\label{sec:limit-licq}

The analysis in this section is performed under Assumption~\ref{ass:basic} and the following two assumptions.  Before stating them, we remark  that all of the results from Section~\ref{sec:no-cq} still hold.

\begin{assumption}\label{ass:X-bounded}
The set $\Xcal$ in Assumption~\ref{ass:basic} is bounded.
\end{assumption}

\begin{assumption}\label{ass:limit-licq}
Let $\Lcal$ denote the set of limit points of the sequence $\{x_k\}$ generated by Algorithm~\ref{alg:main}. Every $x_*\in\Lcal$ satisfies the LICQ, i.e., if $x_*\in \Lcal$, then $[J(x_*)^T, I_{\Acal(x_*)}^T]^T$ has full row rank with $I_{\Acal(x_*)}$ denoting the subset of the rows of the identity matrix $I$ that corresponds to the index set $\Acal(x_*) := \{i \in [n]: [x_*]_i = 0\}$.
\end{assumption} 

The previous assumption has important consequences in terms of a certain type of infeasible point (see Lemma~\ref{lem:opt-equivalence}(ii)), as we now define.


\begin{definition}
We say that $\xbar\in\R{n}$ is an \emph{infeasible stationary point (ISP)} for problem~\eqref{prob:general} if and only if $\xbar\in\Omega$, $\xbar = \Proj_{\Omega}(\xbar - J(\xbar)^T c(\xbar))$, and $c(\xbar) \neq 0$.
\end{definition}

We now show that any limit point of the sequence of iterates cannot be an ISP.
\begin{lemma}\label{lem:ISP}
If $x_*$ is a limit point of $\{x_k\}$, then $x_*$ cannot be an ISP.
\end{lemma}
\begin{proof}
Let $x_*\in\R{n}$ be a limit point of $\{x_k\}$.  Suppose that $x_*\in\Omega$ and $x_* = \Proj_{\Omega}(x_* - J(x_*)^T c(x_*))$.  The proof will be complete if we can show that $c(x_*) = 0$ since this would prove that $x_*$ is not an ISP.  Thus, we now prove that $c(x_*) = 0$.

It follows using the same proof as in Lemma~\ref{lem:opt-equivalence} with $x_k$ replaced by $x_*$ that $x_* = \Proj_{\Omega}(x_* - J(x_*)^T c(x_*))$ implies that $x_*$ is a first-order KKT point for the feasibility problem~\eqref{prob:isp}.  Therefore, there exists $z_*\in\R{n}_{\geq 0}$ satisfying $x_*\cdot z_* = 0$ (componentwise), and $J(x_*)^T c(x_*) = z_*$. It follows from these equations and $\Ical(x_*) = [n]\setminus\Acal(x_*)$ that $[J(x_*)^T c(x_*)]_{\Ical(x_*)} = 0$, where we also note that $\Ical(x_*) \neq \emptyset$ as a consequence of Assumption~\ref{ass:limit-licq}. Letting $J_{\Ical(x_*)}(x_*)$ denote the columns of $J(x_*)$ that correspond to the indices in $\Ical(x_*)$, it follows from above that
$0 = [J(x_*)^T c(x_*)]_{\Ical(x_*)}
= [J_{\Ical(x_*)}(x_*)]^T c(x_*)$. 
Since $J_{\Ical(x_*)}(x_*)$ must have full row rank (see~\cite[Lemma~2.1.3]{Rob07}), it follows that $c(x_*) = 0$, which completes the proof.
\end{proof}

The next result bounds $\|v_k(1)\|_2$ by the infeasibility of the equality constraints.

\begin{lemma}\label{lem:vk-bound-by-ck}
For all $k\in\N{}$, it holds that $\|v_k(1)\|_2 \leq \kappaJ\|c_k\|_2$.
\end{lemma}
\begin{proof}
It follows from the definition of $v_k(1)$ in~\eqref{eq:beta-condition}, $x_k\in\Omega$ for all $k\in\N{}$ by how Algorithm~\ref{alg:main} is designed, non-expansivity of the projection operator, and~\eqref{eq:consequence-of-basic-assumption} that
\begin{align*}
\|v_k(1)\|_2 
&= \|\Proj_{\Omega}(x_k - J_k^Tc_k) - x_k\|_2
= \|\Proj_{\Omega}(x_k - J_k^Tc_k) - \Proj_{\Omega}(x_k)\|_2 \\
&\leq \|J_k^Tc_k\|_2
\leq \kappaJ\|c_k\|_2,
\end{align*}
which completes the proof.
\end{proof}

We can now prove that our infeasiblity measure converges to zero.

\begin{lemma}\label{lem:vk-to-zero-true-limit}
The iterate sequence $\{x_k\}$ satisfies $\lim_{k\to\infty} \|v_k(1)\|_2 = 0$.    
\end{lemma}
\begin{proof}
From Theorem~\ref{thm:complexity}, it follows that there exists a subsequence $\Kcal_1 \subseteq \N{}$ 
such that $\lim_{k \in \Kcal_1} \|v_k(1)\|_2 = 0$. Now, for the purpose of reaching a contradiction, assume that there exists a subsequence of iterations $\Kcal_2\subseteq\N{}\setminus\Kcal_1$ and a scalar $v_{\min}\in\R{}_{>0}$ such that $\|v_k(1)\|_2 \geq v_{\min}$ for all $k\in \Kcal_2$. We now proceed by considering two cases. \\[0.4em]
\textbf{Case 1: $\{\tau_k\} \to 0$.}  The definitions of $\Kcal_1$ and $\Kcal_2$ allow us to define, for each $k\in\Kcal_1$, the quantity $\khat(k)$ as the smallest iteration in $\Kcal_2$ that is strictly larger than $k$. We can use this definition, Lemma~\ref{lem:vk-bound-by-ck}, $\{\tau_k\} \to 0$, \eqref{eq:consequence-of-basic-assumption},  Lemma~\ref{lem:taubar-props}(iii), and nonnegativity of $r$ to conclude that the following holds for each sufficiently large $k\in\Kcal_1$: 
\begin{align*}
\frac{v_{\min}}{2\kappaJ} &\leq \frac{\|c(x_{\khat(k)})\|_2}{2} \leq \tau_{\khat(k)}\big(f_{\khat(k)} -f_{\inf} + r(x_{\khat(k)})\big) + \|c(x_{\khat(k)})\|_2 = \Phibar_{\tau_{\khat(k)}}(x_{\khat(k)}) \\
&\leq \Phibar_{\tau_{k}}(x_{k}) = \tau_{k}\big(f_{k} -f_{\inf} + r(x_{k})\big) + \|c(x_{k})\|_2 \leq 2 \|c_k\|_2.
\end{align*}
It follows from this inequality and the definition of $\Kcal_1$ that 
$$
\lim_{k\in\Kcal_1} \|v_k(1)\|_2 = 0
\ \text{and} \ \ 
\liminf_{k\in\Kcal_1}  \|c_k\|_2 \geq \frac{v_{\min}}{2\kappaJ} > 0.
$$
Therefore, every limit point of $\{x_k\}_{k\in\Kcal_1}$ must be an ISP, and at least one such limit point must exist as a consequence of Assumption~\ref{ass:X-bounded}.  This contradicts Lemma~\ref{lem:ISP}.\\[0.4em]
\noindent\textbf{Case 2: $\{\tau_k\}$ is bounded away from zero.} In this case, it follows from Theorem~\ref{thm:no-licq}(i) that the proximal parameter sequence $\{\alpha_k\}$ is also bounded away from zero.  Given the manner in which both sequences are defined in Algorithm~\ref{alg:main}, we can conclude that there exists $\khat\in\N{}$ such that $\tau_k = \tau_{\khat} > 0$ and $\alpha_k = \alpha_{\khat} > 0$ for all $k \geq \khat$. We may now use the same logic as in the proof of Lemma~\ref{thm:no-licq}(i) and  \eqref{eq:consequence-of-basic-assumption} to obtain
\begin{equation*}
\begin{aligned}
\infty &>  \Phibar_{\tau_0}(x_0) \geq \sum_{k=0}^{\infty}(\Phibar_{\tau_k}(x_k) - \Phibar_{\tau_{k+1}}(x_{k+1})) \\
&\geq  \sum_{\khat \leq k \in \Scal}(\Phibar_{\tau_k}(x_k) - \Phibar_{\tau_{k+1}}(x_{k+1})) \\
&\geq  \sum_{\khat \leq k \in \Scal} \eta_\Phi
\tfrac{\sigma_c\kappa_1}{\kappa_c}\alpha_{\khat}\|v_k(1)\|_2^2 \min \left\{\tfrac{1}{1+ \kappaJ^2}, \kappa_v\alpha_{\khat} \right\},
\end{aligned}
\end{equation*}
which implies that $\lim_{k \in \Scal} \|v_k(1)\|_2 = 0$. Combining this result with the fact that $x_{k+1} = x_k$ whenever $k\notin\Scal$ and that the definition of $v_k(1)$ depends only on $x_k$, the projection onto $\Omega$ (which is continuous), and the continuous functions $c$ and $J$, it follows that $\lim_{k\to\infty} \|v_k(1)\|_2 = 0$.  This contradicts the definition of $\Kcal_2$.

Since we have shown that both \textbf{Case 1} and \textbf{Case 2} cannot occur, and these are the only cases that can possibly occur, we must conclude that our original assumption was incorrect, namely the existence of the set $\Kcal_2$.  This completes the proof. 
\end{proof}

Next, we formally establish that $\Lcal$ is a compact set. 
\begin{lemma}\label{lem:limit-point-set-compact}
The set $\Lcal$ in Assumption~\ref{ass:limit-licq} is compact.
\end{lemma}
\begin{proof}
By Assumption~\ref{ass:X-bounded}, the set $\Lcal$ is bounded. It remains to show that $\Lcal$ is closed. To this end, suppose that $\{x^\Lcal_j\}_{j\geq 1} \subseteq\Lcal$ and $x^\Lcal\in\R{n}$ satisfy $\lim_{j\to\infty} x^\Lcal_j = x^\Lcal$; we prove that $x^\Lcal\in\Lcal$. Let us define a sequence $\Kcal = \{k_1,k_2,\dots\} \subseteq\N{}$.  In particular, let $k_1$ be the smallest integer such that the iterate $x_{k_1}$ satisfies $\|x^\Lcal_1 - x_{k_1}\|_2 \leq 1$.  We then iteratively define $k_j$ for $j \geq 2$ as the smallest integer $k_j$ such that $k_j > k_{j-1}$ and the iterate $x_{k_j}$ satisfies $\|x^\Lcal_j - x_{k_j}\|_2 \leq 1/j$. In summary, $\Kcal =\{k_1,k_2,\dots\} \subseteq\N{}$ is a strictly monotonically increasing subsequence of $\N{}$ such that $
\|x^\Lcal_j - x_{k_j}\|_2 \leq 1/j$ for all $j$. It follows from this inequality and the triangle inequality that
$$
\|x^\Lcal - x_{k_j}\|_2
\leq \|x^\Lcal - x^\Lcal_j\|_2 + \|x^\Lcal_j - x_{k_j}\|_2
\leq  \|x^\Lcal - x^\Lcal_j\|_2 + \tfrac{1}{j} \ \text{for all $j\geq 1$.}
$$
Combining this inequality with $\lim_{j\to\infty} x^\Lcal_j = x^\Lcal$, it follows that $\lim_{j\to\infty} x_{k_j} = x^\Lcal$, which proves that $x^\Lcal\in\Lcal$ as claimed, thus completing the proof.
%
\end{proof}

The next key lemma uses the function $\delta(\cdot): \R{n}\to\R{}_{> 0}$ defined as
\begin{equation}\label{def:deltamin}
\deltamin(x) := \min_{i \in \Ical(x)} [x]_i,
\end{equation}
which gives a measure for how far the inactive variables at $x$ are from being active.


\begin{lemma}\label{lem:uniform-sigma}
The following hold for the set of limit points $\Lcal$:
\begin{itemize}
\item[(i)] There exist $\nlim \in\N{}$, $\{x^\Lcal_i\}_{i=1}^\nlim \subseteq\Lcal$, and $\{\epsilon^\Lcal_i\}_{i=1}^\nlim \subset\R{}_{>0}$ such that
  \begin{itemize}
  \item[(a)] $\Lcal \subset \cup_{i=1}^\nlim \Bcal(x^\Lcal_i,\epsilon^\Lcal_i)$, and
  \item[(b)] if, for some $j$, it holds that $x\in\Bcal(x^\Lcal_j,\epsilon^\Lcal_j)$, then 
  \begin{subequations}
  \begin{align}
  \|x-x^\Lcal_j\|_2 &\leq \tfrac{1}{3}\deltamin(x^\Lcal_j), \label{cond1-xlj}\\
  \Acal(x) &\subseteq\Acal(x^\Lcal_j),
  \ \text{and}\label{cond2-xlj} \\
  \sigmamin\big([J(x)^T, I_{\Acal(x^\Lcal_j)}^T]^T\big)
&\geq \thalf\sigmamin\big([J(x^\Lcal_j)^T, I_{\Acal(x^\Lcal_j)}^T]^T\big). \label{cond3-xlj}
  \end{align}
  \end{subequations}
  \end{itemize}
\item[(ii)] For the objects in part (i), there exists $\epsminL\in\R{}_{>0}$ such that if $\xbar\in\R{n}$ satisfies $\dist(\xbar,\Lcal) \leq \epsminL$, then $\xbar \in \cup_{i=1}^\nlim \Bcal(x^\Lcal_i,\epsilon^\Lcal_i)$ and there exists $j\in[\nlim]$ such that
\begin{equation*}
\sigmamin\big([J(\xbar)^T, I_{\Acal(x^\Lcal_j)}^T]^T\big)
\geq
\min_{i \in [\nlim]}
\thalf\sigmamin\big([J(x^\Lcal_i)^T,I_{\Acal(x^\Lcal_i)}^T]^T\big)
=: \sigmamin^\Lcal > 0,
\end{equation*}
where the inequality $\sigmamin^\Lcal > 0$ is a consequence of Assumption~\ref{ass:limit-licq}.
\end{itemize}
\end{lemma}
\begin{proof}
For $x^\Lcal\in \Lcal$, let $\epsilon(x^\Lcal)\in\R{}_{>0}$ satisfy that if $x\in\Bcal(x^\Lcal,\epsilon(x^\Lcal))$ then $\Ical(x^\Lcal)\subseteq\Ical(x)$, $\|x-x^\Lcal\|_2 \leq \tfrac{1}{3}\deltamin(x^\Lcal)$, and $\sigmamin\big([J(x)^T, I_{\Acal(x^\Lcal)}^T]^T\big) \geq \tfrac{\sigmamin}{2}\big([J(x^\Lcal)^T, I_{\Acal(x^\Lcal)}^T]^T\big)$, where satisfying the third condition is possible because of the continuity of singular values of a matrix with respect to its entries and Assumption~\ref{ass:limit-licq}. It follows that $\cup_{x^\Lcal\in\Lcal} \Bcal(x^\Lcal,\epsilon(x^\Lcal))$ is an open cover of the compact set $\Lcal$ (see Lemma~\ref{lem:limit-point-set-compact}). Using this fact and the definition of a compact set, it follows that there exists a finite subcover, i.e., there exist $\nlim \in\N{}$, $\{x^\Lcal_i\}_{i=1}^\nlim \subseteq\Lcal$, and $\{\epsilon^\Lcal_i\}_{i=1}^\nlim \subset\R{}_{>0}$ such that $\Lcal \subset \cup_{i=1}^\nlim \Bcal(x^\Lcal_i,\epsilon^\Lcal_i)$ and  if, for some $j\in\{1,2,\dots,\nlim\}$, it holds that $x\in\Bcal(x^\Lcal_j,\epsilon^\Lcal_j)$ then 
$\Ical(x^\Lcal_j) \subseteq \Ical(x)$, $\|x-x^\Lcal_j\|_2 \leq \tfrac{1}{3}\deltamin(x^\Lcal_j)$, and $\sigmamin\big([J(x)^T, I_{\Acal(x^\Lcal_j)}^T]^T\big)
\geq \thalf\sigmamin\big([J(x^\Lcal_j)^T, I_{\Acal(x^\Lcal_j)}^T]^T\big)$. Since $\Ical(x^\Lcal_j) \subseteq \Ical(x)$ is equivalent to  $\Acal(x)\subseteq\Acal(x^\Lcal_j)$, we have completed the proof of part (i).

We now prove part (ii).  First, using the \emph{finite} subcover computed in part (i) and the fact that $\Lcal$ is compact, there exists $\epsminL\in\R{}_{>0}$ such that if $x\in\R{n}$ satisfies  $\dist(x,\Lcal)  \leq \epsminL$, then $x\in \cup_{i=1}^\nlim \Bcal(x^\Lcal_i,\epsilon^\Lcal_i)$. Let $\xbar$ be an arbitrary point that satisfies $\dist(\xbar,\Lcal) \leq \epsminL$.  Then, it follows that  there exists $j\in\{1,2,\dots,\nlim\}$ such that $\xbar\in\Bcal(x^\Lcal_j,\epsilon^\Lcal_j)$, which combined with part (i)(b) gives $\Acal(\xbar)\subseteq\Acal(x^\Lcal_j)$ and $\sigmamin\big([J(\xbar)^T, I_{\Acal(x^\Lcal_j)}^T]^T\big)
\geq \thalf\sigmamin\big([J(x^\Lcal_j)^T,I_{\Acal(x^\Lcal_j)}^T]^T\big) \geq \sigmamin^\Lcal > 0$, as claimed.
%
\end{proof}

The next result shows that iterates of the algorithm eventually satisfy the properties of the previous lemma.

\begin{lemma}\label{lem:results-at-xk}
There exists $\kbar\in\N{}$ such that, for each $k\geq \kbar$, there exists a corresponding $j\in[\nlim]$ that satisfies, with $\sigmaminL$ defined in Lemma~\ref{lem:uniform-sigma}(ii), the following:
\begin{subequations}
\begin{align}
  \|x_k-x^\Lcal_j\|_2 &\leq \tfrac{1}{3}\deltamin(x^\Lcal_j), \label{cond1-xlj-xk}\\
  \Acal(x_k) &\subseteq\Acal(x^\Lcal_j),
  \ \text{and}\label{cond2-xlj-xk} \\
  \sigmamin\big([J_k^T, I_{\Acal(x^\Lcal_j)}^T]^T\big)
&\geq \thalf\sigmamin\big([J(x^\Lcal_j)^T, I_{\Acal(x^\Lcal_j)}^T]^T\big) \geq \sigmaminL > 0. \label{cond3-xlj-xk}
\end{align}
\end{subequations}
\end{lemma}
\begin{proof}
 Let $\epsminL > 0$ be defined as in Lemma~\ref{lem:uniform-sigma}(ii).  Since  $\Lcal$ is the set of all limit points, there exists an iteration $\kbar$ such that $\dist(x_k,\Lcal) \leq \epsminL$ for all $k\geq \kbar$ (this $\kbar$ is now the $\kbar$ whose existence is claimed in the statement of the current lemma). For the remainder of the proof, consider arbitrary $k\geq \kbar$.  It follows from the definition of $\kbar$ that $\dist(x_k,\Lcal) \leq \epsminL$, and then from  Lemma~\ref{lem:uniform-sigma}(ii) that there exists $j\in[\nlim]$ such that $x_k\in\Bcal(x^\Lcal_j,\epsilon^\Lcal_j)$. Conditions~\eqref{cond1-xlj-xk}--\eqref{cond3-xlj-xk} now follow from Lemma~\ref{lem:uniform-sigma}.  
\end{proof}

We now give a lower bound on $\|v_k(1)\|_2$ in terms of $\|c_k\|_2$, which is crucial to giving a lower bound on the merit parameter sequence. The result uses the constant
\begin{equation}\label{def:dletaminL}
\deltaminL:= \min_{j\in[\nlim]} \deltamin(x^\Lcal_j) > 0.
\end{equation}

\begin{lemma}\label{lem:vk-ck-relation-limit-licq}
For all sufficiently large $ k \in \mathbb{N}$, it holds that $\|v_k(1)\|_2 \geq \sigmamin^\Lcal\|c_k\|_2$, where the positive constant $\sigmaminL$ is defined in Lemma~\ref{lem:uniform-sigma}(ii).
\end{lemma}
\begin{proof}
 With $\deltaminL$ in~\eqref{def:dletaminL}, Lemma~\ref{lem:vk-to-zero-true-limit} ensures the existence $\kbar_1$ such that
\begin{equation}\label{eq:large-k-vk(1)}
\|v_k(1)\|_2 =  \bigl\|\Proj_{\Omega}(x_k - J_k^T c_k) - x_k\bigr\|_2
\le \tfrac{1}{3}\deltaminL \ \ \text{for all $k\geq\kbar_1$.}
\end{equation}
Let $\{\epsminL,\sigmaminL\}\subset\R{}_{>0}$ be as stated in Lemma~\ref{lem:uniform-sigma}, and let $\kbar_2$ play the role of $\kbar$ from Lemma~\ref{lem:results-at-xk}.  
For the remainder of the proof, consider arbitrary $k\geq \max\{\kbar_1,\kbar_2\}$.  It follows from the definition of $\kbar_2$ that $x_k$ satisfies~\eqref{cond1-xlj-xk}--\eqref{cond3-xlj-xk} for some $j\in[\nlim]$.
Using~\eqref{eq:large-k-vk(1)}, \eqref{cond1-xlj-xk}, and definitions of $\deltamin(x^\Lcal_j)$ and $\deltaminL$, each $i \in \Ical(x^\Lcal_j)$ satisfies
\begin{equation*}
\begin{aligned}
\relax[\Proj_{\Omega}(x_k - J_k^T c_k)]_i 
&\geq [x_k]_i - \tfrac{1}{3}\deltaminL 
 \geq [x^\Lcal_j]_i - \tfrac{1}{3}\deltamin(x^\Lcal_j) - \tfrac{1}{3}\deltaminL \\
&\geq \delta_{\min}(x^\Lcal_j) - \tfrac{1}{3}\delta_{\min}(x^\Lcal_j)  - \tfrac{1}{3}\deltaminL 
= \tfrac{2}{3} \delta_{\min}(x^\Lcal_j) - \tfrac{1}{3}\deltaminL \\ 
&\geq \tfrac{2}{3} \deltaminL - \tfrac{1}{3}\deltaminL = \tfrac{1}{3}\deltaminL.
\end{aligned}
\end{equation*}
Hence, for all $i \in \Ical(x^\Lcal_j)$ it holds that $[x_k - J_k^T c_k]_i > 0$. Now, define $w_k \in \R{n}$ as
\begin{equation}\label{eq:multiplier-new}
[w_k]_i =
\begin{cases}
0 & \text{if $i \in \Ical(x^\Lcal_j)$,} \\[3pt]
-[J_k^T c_k]_i - [v_k(1)]_i & \text{if $i \in \Acal(x^\Lcal_j)$.}
\end{cases}
\end{equation}
The definition of $w_k$,  the fact that $[x_k - J_k^T c_k]_i > 0$ for all $i \in \Ical(x^\Lcal_j)$, and~\eqref{cond3-xlj-xk}
give 
\begin{equation*}
\begin{aligned}
\left\|v_k(1)\right\|_2 
&= \left\|\Proj_{\Omega} (x_k - J_k^Tc_k) - x_k\right\|_2
= \left\|-J_k^Tc_k-w_k\right\|_2 \\
&= \left\|\begin{bmatrix}
    J_k^T, I_{\Acal(x^\Lcal_j)}^T 
\end{bmatrix} 
\begin{bmatrix}
    c_k \\
    [w_k]_{\Acal(x^\Lcal_j)}
\end{bmatrix}\right\|_2 \geq \sigmamin^\Lcal \|c_k\|_2,
\end{aligned} 
\end{equation*}
which completes the proof.
\end{proof}

Our next result gives a new bound on the model decrease. 
\begin{lemma}\label{lem:cauchy-decrease-another}
For $\kappa_1 \in (0,1]$ in Lemma~\ref{lem:cauchy-decrease-projection}, all sufficiently large $k\in\N{}$ satisfy
\begin{equation}\label{diff-m-no-assumptions-another}
\|c_k\|_2 - \|c_k + J_k v_k^c\|_2 
\geq \kappa_1(\sigmamin^\Lcal)^2\|c_k\|_2 \min \left\{\frac{1}{1+ \kappaJ^2}, \kappa_v \alpha_k \right\}.
\end{equation}
\end{lemma}
\begin{proof}
If $\delta_k = 0$, then either $\|c_k\|_2 = 0$ and the inequality holds trivially, or $\|c_k\|_2 \neq 0$ and the algorithm terminates finitely, which is a contradiction to our overall setting in this subsection that the algorithm does not terminate finitely.  Therefore, we may proceed assuming $\delta_k \neq 0$.  It follows from Lemma~\ref{lem:cauchy-decrease-projection} that  $\|c_k+J_kv_k^c\|_2 \leq \|c_k\|_2$.  Using this inequality and a  difference-of-squares computation, we have that
\begin{equation}~\label{eq:feasibility-reduction-inequality-new-another}
\begin{aligned}
\|c_k\|_2^2 - \|c_k + J_k v_k^c\|_2^2 &= (\|c_k\|_2 + \|c_k + J_k v_k^c\|_2)(\|c_k\|_2 - \|c_k + J_k v_k^c\|_2) \\
&\leq 2\|c_k\|_2 (\|c_k\|_2 - \|c_k + J_k v_k^c\|_2).
\end{aligned}
\end{equation}
Combining~\eqref{eq:feasibility-reduction-inequality-new-another}, \eqref{eq:cauchy-decrease-projection-lemma}, Lemma~\ref{lem:vk-ck-relation-limit-licq}, and~\eqref{eq:consequence-of-basic-assumption}, all sufficiently large $k\in\N{}$ satisfy
\begin{align*}
2\|c_k\|_2(\|c_k\|_2 - \|c_k + J_k v_k^c\|_2) 
&\geq \|c_k\|_2^2 - \|c_k + J_k v_k^c\|_2^2 
= 2(m_k(0) - m_k(v_k^c)) \\ 
&\geq 2\kappa_1 \|v_k(1)\|_2^2 \min \left\{\frac{1}{1+ \|J_k^T J_k\|_2}, \kappa_v \alpha_k\right\} \\
&\geq 2\kappa_1(\sigmamin^\Lcal)^2 \|c_k\|_2^2 \min \left\{\frac{1}{1+ \kappaJ^2}, \kappa_v \alpha_k\right\}.
\end{align*}
If $\|c_k\|_2 = 0$, then again the desired inequality holds trivially.  Otherwise, dividing the above inequality by $2\|c_k\|_2$  gives~\eqref{diff-m-no-assumptions-another}, and thus completes the proof.
\end{proof}

We now bound the merit and proximal parameter sequences away from zero.

\begin{lemma}\label{lem:tau-bounded-another}
Let $\kbar>0$ be sufficiently large that the results in Lemma~\ref{lem:vk-ck-relation-limit-licq} and Lemma~\ref{lem:cauchy-decrease-another} hold.  Then, each $k\geq \kbar$ yields 
\begin{align}
\tauktrial &\geq \taubar_{\min,trial} := \frac{(1-\sigma_c)\kappa_1(\sigmamin^\Lcal)^2 \min \left\{\frac{1}{(1+ \kappaJ^2)\alpha_0} , \kappa_v\right\}}{2(\kappag + \kappa_{\partial r}) \kappav \kappaJ + 2\kappav^2 \kappaJ^2 \kappac} > 0.
\label{bound-tautrial-another-large}
\end{align}
The merit parameter sequence itself satisfies, for all $k\in\N{}$, the inequality 
\begin{align}
\tau_k &\geq \taubar_{\min}:=\min\{\tau_{\kbar-1}, (1-\epsilon_{\tau})\taubar_{\min,trial})\} > 0.\label{bound-tau-another-large}
\end{align}
Finally, the proximal parameter sequence satisfies, for all $k\in\N{}$, the inequality
\begin{equation}\label{alpha-bounded-limit-case}
\alpha_k \geq \alphabarmin 
:= \min\{\alpha_0,\tfrac{\xi\taubar_{\min}}{2(\taubar_{\min}\lipg+\lipJ)}\}
> 0. 
\end{equation}
\end{lemma}
\begin{proof}
We first prove~\eqref{bound-tautrial-another-large}. If $A_k \leq 0$ in the definition of $\tauktrial$, then $\tauktrial = \infty$ so that~\eqref{bound-tautrial-another-large} trivially holds.  If $A_k > 0$, then it follows from the definition of $\tauktrial$, $s_k = v_k + u_k$, $J_k u_k = 0$ (see~\eqref{kkt:nullspace}), Lemma~\ref{lem:denominator},  Lemma~\ref{lem:cauchy-decrease-another}, the fact that $\alpha_k \leq \alpha_0$ for all $k$ by construction of Algorithm~\ref{alg:main}, and~\eqref{eq:consequence-of-basic-assumption} that each $k\geq \kbar$ yields
\begin{align*}
\tau_{k,\text{trial}} 
&= \frac{(1-\sigma_c)(\|c_k\|_2 - \|c_k + J_k v_k\|_2)}{g_k^T s_k + \tfrac{1}{2\alpha_k}\|s_k\|_2^2 + r(x_k+s_k) - r_k} \\
&\geq \frac{(1-\sigma_c)\kappa_1(\sigmamin^\Lcal)^2\|c_k\|_2 \min \left\{\frac{1}{1+ \kappaJ^2}, \kappa_v \alpha_k \right\}}{2(\kappag + \kappa_{\partial r}) \kappav \kappaJ \alpha_k \|c_k\|_2 + 2\kappav^2 \kappaJ^2 \kappac \alpha_k \|c_k\|_2} \\
&\geq \frac{(1-\sigma_c)\kappa_1(\sigmamin^\Lcal)^2 \min \left\{\frac{1}{(1+ \kappaJ^2)\alpha_0} , \kappa_v\right\}}{2(\kappag + \kappa_{\partial r}) \kappav \kappaJ + 2\kappav^2 \kappaJ^2 \kappac},
\end{align*}
which proves that~\eqref{bound-tautrial-another-large} holds for all $k\geq \kbar$, as claimed.  The merit parameter update rule~\eqref{eq:tau-update} and~\eqref{bound-tautrial-another-large} give~\eqref{bound-tau-another-large}.
Finally, \eqref{alpha-bounded-limit-case} follows from~\eqref{bound-tau-another-large} and Lemma~\ref{lem:alpha-bounded}.
\end{proof}



The next result establishes that the norm of the search direction converges to zero along the sequence of successful iterations.
\begin{lemma}\label{lem:sk-to-zero}
The search direction sequence $\{s_k\}_{k \in \Scal}$ satisfies $\lim_{k \in \Scal} \|s_k\|_2 = 0$.
\end{lemma}
\begin{proof}
We first note that the derivation of \eqref{eq:merit-function-reduction-old} still holds under the assumptions of this section, and therefore we know that 
\begin{equation}~\label{eq:merit-function-reduction-stationarity}
\Phibar_{\tau_k}(x_k) - \Phibar_{\tau_{k+1}}(x_{k+1}) \geq \sum_{k \in \Scal} \eta_\Phi \tfrac{\tau_k}{8\alpha_k} \|s_k\|_2^2.
\end{equation}
Using nonnegativity of $\Phibar_\tau$ 
in~\eqref{eq:shifted-merit-function}, 
Lemma~\ref{lem:taubar-props}(ii)-(iii), and~\eqref{eq:merit-function-reduction-stationarity}, we have that 
\begin{equation*}
\begin{aligned}
\infty &> 
\sum_{k \in \Scal}(\Phibar_{\tau_k}(x_k) - \Phibar_{\tau_{k+1}}(x_{k+1})) \geq \sum_{k \in \Scal} \eta_\Phi \tfrac{\tau_k}{8\alpha_k} \|s_k\|_2^2.
\end{aligned}
\end{equation*}
Lemma~\ref{lem:tau-bounded-another} gives $\tau_k \geq \taubar_{\min} > 0$ for all $k\in\N{}$, where $\taubar_{\min}$ is defined in~\eqref{bound-tau-another-large}, so that
$\sum_{k \in \Scal} \eta_{\Phi} \tfrac{\taubar_{\min}}{8\alpha_0} \|s_k\|_2^2 < \infty$,
which implies $\lim_{k \in \Scal} \|s_k\|_2 = 0$, and completes the proof.
\end{proof}

We next prove that the sequence of Lagrange multiplier estimates generated by subproblem~\eqref{subprob:stationarity} during successful iterations are bounded.
\begin{lemma}\label{lem:multiplier-bounded}
There exists $\kappayz\in\R{}_{>0}$ so that $\max_{k\in\Scal}\max\{\|y_k\|_\infty,\|z_k\|_\infty\} \leq \kappayz$.
\end{lemma}
\begin{proof}
Let $\kbar_1$ serve the role of $\kbar$ in Lemma~\ref{lem:results-at-xk} so that the results of Lemma~\ref{lem:results-at-xk} hold for each $k \geq \kbar_1$.  Let $\kbar_2$ be sufficiently large so that $\|s_k\|_2 \leq \tfrac{1}{3}\deltaminL$ for all $\kbar_2 \leq k\in\Scal$, which is possible because of how $\deltaminL$ is defined and Lemma~\ref{lem:sk-to-zero}.

For the remainder of the proof, consider an arbitrary $k$ with $\max\{\kbar_1,\kbar_2\} \leq k\in\Scal$.  Let $j\in[\nlim]$ be the value guaranteed by Lemma~\ref{lem:results-at-xk} to exist so~\eqref{cond1-xlj-xk}--\eqref{cond3-xlj-xk} hold. 

Next, consider $i\in\Ical(x^\Lcal_j)$.  It follows from~\eqref{cond1-xlj-xk}, the triangle inequality, the definition of $\kbar_2$, and the definition of $\deltamin^\Lcal$ (see~\eqref{def:dletaminL}) that 
$$
\|x_k+s_k - x^\Lcal_j\|_2 
\leq \|x_k-x^\Lcal_j\|_2 + \|s_k\|_2
\leq \tfrac{1}{3}\deltamin(x^\Lcal_j) + \tfrac{1}{3}\deltamin^\Lcal
\leq \tfrac{2}{3}\deltamin(x^\Lcal_j).
$$ 
This inequality, the definition of $\deltamin(x^\Lcal_j)$ (see~\eqref{def:deltamin}), and $i\in\Ical(x^\Lcal_j)$ imply that 
$$
[x_k+s_k]_i 
\geq [x^\Lcal_j]_i - \tfrac{2}{3}\deltamin(x^\Lcal_j)
\geq \deltamin(x^\Lcal_j) - \tfrac{2}{3}\deltamin(x^\Lcal_j)
= \tfrac{1}{3}\deltamin(x^\Lcal_j) > 0,
$$
so that $i\in\Ical(x_k+s_k)$.  Thus,  $\Ical(x^\Lcal_j)\subseteq\Ical(x_k+s_k)$, or equivalently  $\Acal(x_k+s_k)\subseteq\Acal(x^\Lcal_j)$. 

Now, let us introduce the notation $\Acal^s_k = \Acal(x_k+s_k)$. It follows from $s_k = v_k + u_k$, \eqref{kkt:stationary}, $[z_k]_i = 0$ for all $i\notin\Acal^s_k$ (see~\eqref{kkt:complementarity}), and $\Acal^s_k \subseteq \Acal(x^\Lcal_j)$ (see above) that 
\begin{align*}
 g_k + \tfrac{1}{\alpha_k}s_k + g_{r,k}
 = [J_k^T, I_{\Acal_k^s}^T]\begin{bmatrix}
    y_k \\
    (z_k)_{\Acal_k^s}
\end{bmatrix} 
= [J_k^T, I_{\Acal(x^\Lcal_j)}^T]\begin{bmatrix}
    y_k \\
    (z_k)_{\Acal(x^\Lcal_j)}
\end{bmatrix}.
\end{align*}
Combining this result with~\eqref{cond3-xlj-xk} and $\Acal^s_k \subseteq \Acal(x^\Lcal_j)$ it follows that
$$
\left\|g_k + \tfrac{1}{\alpha_k}s_k + g_{r,k}\right\|_2
\geq \sigmaminL 
\left\|
\begin{bmatrix}
    y_k \\
    (z_k)_{\Acal(x^\Lcal_j)}
\end{bmatrix}
\right\|_2
= \sigmaminL 
\left\|
\begin{bmatrix}
    y_k \\
    z_k
\end{bmatrix}
\right\|_2.
$$
Combining this inequality with the triangle inequality, \eqref{eq:consequence-of-basic-assumption}, $\|s_k\|_2 \leq \tfrac{1}{3}\deltaminL$, and $\alpha_k \geq \alphabarmin$ (see~\eqref{alpha-bounded-limit-case}) it follows that 
$$
\left\|
\begin{bmatrix}
    y_k \\
    z_k
\end{bmatrix}
\right\|_2
\leq \tfrac{1}{\sigmaminL}(\kappag  + \tfrac{\deltaminL}{3\alphabar_{\min}} +
\kappagr).
$$
Since the right-hand side of this inequality is a constant and independent of $k$, we know that the sequence of Lagrange multipliers over the successful iterations is bounded.
\end{proof}




\begin{theorem}\label{thm:KKT}
Let Assumption~\ref{ass:basic} and Assumption~\ref{ass:limit-licq} hold. Any limit point $x_*$ of the sequence $\{x_k\}_{k \in \Scal}$ is a first-order KKT point for problem~\eqref{prob:general}.
\end{theorem}
\begin{proof}
Let $\xstar$ be a limit point of $\{x_k\}_{k \in \Scal}$, i.e., there exists infinite $\Kcal_1 \subseteq \Scal$ satisfying $\{x_k\}_{k\in\Kcal_1}\to\xstar$. From Lemma~\ref{lem:vk-to-zero-true-limit} and Lemma~\ref{lem:vk-ck-relation-limit-licq}, we have that \begin{equation}\label{eq:ck-to-zero}
0 = \lim_{k \to \infty} \|v_k(1)\|_2 \geq \lim_{k \to \infty} \sigmaminL\|c_k\|_2 \geq 0,
\end{equation} 
which implies that $0 = \lim_{k \to \infty} \|c_k\|_2 = \lim_{k\in\Kcal_1} \|c_k\|_2$.  Combining this with continuity of $c$ and $\{x_k\}_{k \in \Scal} \to \xstar$ it follows that $c(\xstar) = 0$.

Next, Lemma~\ref{lem:multiplier-bounded} ensures the existence of a vector pair  $(y_*,z_*) \in \R{m}\times\R{n}$ and infinite subsequence $\Kcal_2\subseteq \Kcal_1$ such that $\{(y_k,z_k)\}_{k\in\Kcal_2} \to (y_*,z_*)$. Also, it follows from Lemma~\ref{lem:sk-to-zero} and Lemma~\ref{lem:snorm-bound} that
$$
0 = \lim_{k\in\Kcal_2}\|s_k\|_2 
\geq \lim_{k\in\Kcal_2} \|\min\{x_k,-z_k\}\|_2 \geq 0,
$$
which implies that $\lim_{k\in\Kcal_2} \|\min\{x_k,-z_k\}\|_2 = 0$. Combining this with the continuity of the min operator and $\{(y_k,z_k)\}_{k \in \Kcal_2} \to (y_*,z_*)$ it follows that $\min\{\xstar,-z_*\} = 0$.

It follows from Lemma~\ref{lem:sk-to-zero} and~\eqref{alpha-bounded-limit-case} that $\lim_{k\in\Kcal_2} (1/\alpha_k)\|s_k\|_2 = 0$.  This fact, \eqref{kkt:stationary}, $\{(x_k,y_k,z_k)\}_{k\in\Kcal_2} \to (\xstar,y_*,z_*)$, and continuity of $g$ and $J$ give
$$
g_{r,*} := -g(x_*) - J(x_*)^T y_* - z_* = \lim_{k\in\Kcal_3} ( -g_k - J_k^T y_k - z_k ) = \lim_{k\in\Kcal_3} g_{r,k},
$$
so that $g(x_*) + g_{r,*} + J(x_*)^T y_* + z_* = 0$.  It follows from this equality, $c(\xstar) = 0$, and $\min\{\xstar,-z_*\} = 0$ that $x_*$ is a first-order KKT point for problem~\eqref{prob:general}, as claimed.
\end{proof}

\subsection{Active set Identification}\label{subsec:active.set}
Our result in this section shows, under suitable assumptions, that our method can successfully identify the optimal active set.

\begin{theorem}\label{thm:active-set-id}
Let $\xstar$ be a first-order KKT point for problem~\eqref{prob:general} with Lagrange multiplier vectors $y_*\in\R{m}$ and $z_*\in\R{n}_{\leq 0}$ for the equality constraints and bound constraints, respectively. Suppose that strict complementarity holds, i.e., that $\max\{\xstar,-z_*\} > 0$. Let  $\Scal_1 \subseteq \Scal$ be such that $\{x_k\}_{k \in \Scal_1} \to \xstar$,  $\{s_k\}_{k \in \Scal_1} \to 0$, and $\{z_k\}_{k \in \Scal_1} \to z_*$. Then, $\Acal(x_{k+1}) = \Acal(\xstar)$ for all sufficiently large $k \in \Scal_1$.
\end{theorem}
\begin{proof}
We have from the optimality conditions in~\eqref{kkt:u} that
\begin{equation}\label{sc-again}
\|\min\{x_k+s_k,-z_k\}\|_2 = 0 \ \ \text{for all $k\in\N{}$.}
\end{equation}
It follows from strict complementarity that $\epsilon := \min\{[-z_*]_j: j \in \Acal(\xstar)\}  > 0$.  Combining this with $\{z_k\}_{k\in\Scal_1} \to z_*$ gives the existence of  $\kbar \in\N{}$ such that $\|z_k - z_*\|_{\infty} < \epsilon/2$ for all $\kbar \leq k \in \Scal_1$. Thus, all $\kbar \leq k \in\Scal_1$ and $j \in\Acal(\xstar)$ satisfy $[-z_k]_j > \tfrac{\epsilon}{2}$. Combining this with~\eqref{sc-again} shows that $[x_{k+1}]_i = [x_k + s_k]_i = 0$ for all $\kbar \leq k\in\Scal_1$ and $i \in \Acal(\xstar)$.  Finally, it follows from $\{x_k\}_{k \in \Scal_1} \to \xstar$ and $\{s_k\}_{k \in \Scal_1} \to 0$ that $[x_{k+1}]_i = [x_k + s_k]_i > 0$ for all $i\notin\Acal(x_*)$ and $k\in\Scal_1$ sufficiently large, which completes the proof.
\end{proof}

\subsection{Manifold Identification}

In this section, we establish a manifold identification property for Algorithm~\ref{alg:main} under certain assumptions. For the definition of a $C^2$-smooth manifold $\Mcal\subset\R{n}$ at a given point in $\R{n}$, see~\cite[Definition~2.3]{lewis2013partial}.  Our result assumes that the regularizer $r$ is partly smooth relative to a manifold at a first-order KKT point; see~\cite[Definition~3.2]{lewis2013partial}.



To motivate our assumption that the regularizer is partly smooth, consider $r(x) = \|x\|_{1}$ and  $x_*\in \R{n} \setminus \{0\}$.
Define the set $\Mcal = \{x \in \R{n}: \text{sgn}(x_i) = \text{sgn}([x_*]_i) \text{ for } i \in \Ical(x_*), \ \text{and} \  x_i=0 \text{ for } i \in \Acal(x_*) \}$, which is a $(|\Ical(x_*)|)$-dimensional $C^2$-smooth manifold around the point $x_*$. Then, $r$ is partly smooth at $x_*$ relative to $\Mcal$. 

We are now ready to present our manifold identification property of Algorithm~\ref{alg:main}. The proof borrows ideas from~\cite[Lemma 1]{lee2023accelerating} and relies on~\cite[Theorem 4.10]{lewis2013partial}.
\begin{theorem}\label{thm:manifold}
Let $x_*$ be a first-order KKT point to problem~\eqref{prob:general} with Lagrange multiplier vectors $y_*$ and $z_*$, and suppose that $r$ is convex and partly smooth at $x_*$ relative to a $C^2$-smooth manifold $\Mcal$. Assume that the proximal parameter sequence $\{\alpha_k\}_{k \in \N{}}$ is bounded away from zero, that there exists a subsequence $\Scal_1 \subseteq \Scal$ such that $\{(x_k,s_k,y_k,z_k)\}_{k \in \Scal_1} \to (x_*,0,y_*,z_*)$, and that the non-degeneracy condition \begin{equation}\label{eq:nondegeneracy}
0 \in \{g(x_*) + J(x_*)^Ty_* + z_*\} + \textnormal{relint}(\partial r(x_*))
\end{equation}
holds, where $\text{relint}$ denotes the relative interior of a convex set.  Then, it follows  that $x_{k+1} \in \Mcal$ for all sufficiently large $k \in \Scal_1$. 
\end{theorem}
\begin{proof}
Let us define $\ybar = -( g(x_*) + J(x_*)^T y_* + z_*)$, and note from~\eqref{eq:nondegeneracy} that $\ybar\in\textnormal{relint}(\partial r(x_*))$. Next,  since $r$ is convex, it is prox-regular \cite[Definition~3.6]{lewis2013partial} at $x_*$ with $\ybar$.  It also follows from $r$ being convex (thus continuous), $\{x_k\}_{k\in\Scal_1} \to x_*$, and $\{s_k\}_{k\in\Scal_1} \to 0$ that $\{x_k+s_k\}_{k\in\Scal_1} \to x_*$ and $\{r(x_k+s_k)\}_{k\in\Scal_1} \to r(x_*)$. Combining these observations with the assumption in the statement of the theorem that $r$ is partly smooth at $x_*$ relative to a $C^2$-smooth manifold $\Mcal$, means that every assumption in~\cite[Theorem 4.10]{lewis2013partial} holds (with $r$ and $x_*$ here playing the role of $f$ and $\xbar$ in~\cite[Theorem 4.10]{lewis2013partial}). To use~\cite[Theorem 4.10]{lewis2013partial}) to establish our manifold identification result, it remains to prove that $\{\text{dist}\left(\ybar,\partial r(x_k+s_k)\right)\}_{k\in\Scal_1}\to 0$, as we now show.



It follows from the triangle inequality, \eqref{eq:consequence-of-basic-assumption}, and~\eqref{eq.Lipschitz} that
\begin{equation}\label{eq:util.inequality}
\begin{aligned}
& \ \|J(x_k+s_k)^T y_* - J(x_k)^T y_k\|_2  \\
\leq& \ \|J(x_k+s_k)^T y_* - J(x_k)^T y_* + J(x_k)^T y_* -J(x_k)^T y_k\|_2 \\
\leq& \ L_J \|s_k\|_2\|y_*\|_2 
+ \kappa_J \|y_k - y_*\|_2 \ \ \text{for all $k\in\N{}$.}
\end{aligned}
\end{equation}
Using~\eqref{kkt:stationary}, $g_{r,k} \in \partial r(x_k+s_k)$,~\eqref{eq:consequence-of-basic-assumption}, and~\eqref{eq:util.inequality}, we have that 
\begin{equation*}
\begin{aligned}
& \ \text{dist}\left(-g(x_k+s_k) - J(x_k+s_k)^T y_* - z_*,\,\partial r(x_k+s_k)\right) \\
\leq& \ \left\|-g(x_k+s_k) - J(x_k+s_k)^T y_* - z_* -g_{r,k}\right\|_2 \\
=& \ \Big\| g(x_k+s_k)-g(x_k) + \big(J(x_k+s_k)^T y_* - J(x_k)^T y_k\big) + (z_* - z_k) - \tfrac{1}{\alpha_k}s_k \Big\|_2 \\
\leq& \ \|g(x_k+s_k)-g(x_k)\|_2 + \| J(x_k+s_k)^T y_* - J(x_k)^T y_k \|_2 + \|z_* - z_k\|_2 + \tfrac{1}{\alpha_k} \|s_k\|_2 \\
\leq& \ L_g \|s_k\|_2 + L_J \|s_k\|_2\|y_*\|_2 
+ \kappa_J \|y_k - y_*\|_2 + \|z_k - z_*\|_2 + \tfrac{1}{\alpha_k}\|s_k\|_2 \ \ \text{for all $k\in\N{}$.}
\end{aligned}
\end{equation*}
This inequality, $\{(x_k,s_k,y_k,z_k)\}_{k\in\Scal_1} \to (x_*,0,y_*,z_*)$, and $\{\alpha_k\}$ bounded from $0$ give\!
\begin{equation}\label{eq:dist-equivalent}
\{\text{dist}(-g(x_k+s_k) - J(x_k+s_k)^T y_* - z_*, \partial r(x_k+s_k))\}_{k\in\Scal_1} \to 0.
\end{equation}
Next, for all $k\in\N{}$, it follows from \cite[Theorem~6.2]{CurtRobi25}
that
\begin{align*}
& | \text{dist}(\ybar,\partial r(x_k+s_k)) - \text{dist}(-g(x_k + s_k) - J(x_k+s_k)^T y_* - z_*,\partial r(x_k+s_k))| \\ 
&\leq \|\ybar + g(x_k + s_k) + J(x_k+s_k)^Ty_* + z_*\|_2, 
\end{align*}
which immediately implies that
\begin{align*}
\text{dist}(\ybar,\partial r(x_k+s_k)) 
&\leq \text{dist}(-g(x_k + s_k) - J(x_k+s_k)^T y_* - z_*,\partial r(x_k+s_k)) \\
&\phantom{=} + \|\ybar + g(x_k + s_k) + J(x_k+s_k)^Ty_* + z_*\|_2.
\end{align*}
Combining this inequality  with~\eqref{eq:dist-equivalent}, $\{(x_k,s_k,y_k,z_k)\}_{k\in\Scal_1} \to (x_*,0,y_*,z_*)$, and continuity of $g$ and $J$ shows that $\{\text{dist}\left(\ybar,\partial r(x_k+s_k)\right)\}_{k\in\Scal_1}\to 0$, which was our goal. We can now apply~\cite[Theorem 4.10]{lewis2013partial} to conclude that $x_k+s_k \in\Mcal$ for all sufficiently large $k\in\Scal_1$.  Since $x_{k+1} = x_k+s_k$ for all $k\in\Scal_1$, the proof is completed.
\end{proof}

\section{Numerical Results}\label{sec:numerical}
We present results from numerical experiments conducted using our Python implementation of Algorithm~\ref{alg:main}. The test problems employ the $\ell_1$ regularizer, a widely adopted choice to induce sparse solutions. Our numerical evaluation has two primary objectives: to demonstrate the numerical performance of our method using standard optimization metrics, and to assess its capability to correctly identify the zero-nonzero structure of the solution. Our test problems include special instances of $\ell_1$-regularized optimization problems from the  CUTEst~\cite{gould2015cutest} test environment, and instances of sparse canonical correlation analysis.

\subsection{Implementation details}\label{subsec:implementation}



Given $v_k^c$ in~\eqref{eq:beta-condition} as the Cauchy point for subproblem~\eqref{subprob:feasibility}, to find a $v_k$ satisfying the conditions in~\eqref{eq:vk-condition}, we first compute 
\begin{equation}\label{subprob:reformulate-trust-region}
v_k^\infty := \arg\min_{v\in\R{n}} \ m_k(v) \ \ \text{s.t.} \ \|v\|_{\infty} \leq \kappa_v^{\infty} \alpha_k \delta_k,\  x_k + v \in\Omega 
\end{equation}
with  $\kappa_v^\infty\in\R{}_{>0}$, which differs from~\eqref{subprob:feasibility} only in its use of the infinity-norm.  
Our motivation for using subproblem~\eqref{subprob:reformulate-trust-region} is that the feasible region only consists of simple bound constraints, which can be handled efficiently by solvers. As long as $\kappa_v^\infty \leq \tfrac{1}{\sqrt{n}}\kappav$ (which we choose to hold), the solution $v_k^\infty$ to~\eqref{subprob:reformulate-trust-region} satisfies $\|v_k^\infty\|_2 \leq \sqrt{n}\|v_k^\infty\|_\infty \leq \sqrt{n}\kappa_v^\infty \alpha_k\delta_k \leq \kappav\alpha_k\delta_k$, meaning that $v_k^\infty$ satisfies the first two conditions in~\eqref{eq:vk-condition}. To ensure that the third condition is also satisfied, we set
$$
v_k \gets
\begin{cases}
v_k^c & \text{if $m_k(v_k^c) < m_k(v_k^{\infty})$}, \\
v_k^{\infty} & \text{otherwise.}
\end{cases}
$$
To solve subproblem~\eqref{subprob:reformulate-trust-region}, we use the barrier method in Gurobi version 11.0.3~\cite{gurobi}.

Next, to solve subproblem~\eqref{subprob:stationarity} (as needed in Line~\ref{line:uk} of Algorithm~\ref{alg:main}), we exploit the structure of the $\ell_1$-norm. By introducing variables  $(p,q) \in \R{n}_{\geq 0} \times \R{n}_{\geq 0}$ and using $e$ to denote a ones vector of appropriate dimension, we solve the equivalent problem
\begin{equation}~\label{subprob:stationarity-qp}
\begin{aligned}
\min_{(u,p,q) \in \R{n}\times\R{n}\times\R{n}} & \ g_k^T u + \tfrac{1}{2\alpha_k} \|u\|_2^2 + \tfrac{1}{\alpha_k}v_k^Tu + \lambda e^T(p+q) \\
\text{s.t.} & \ J_k u = 0, \ \ x_k + v_k + u \in \Omega, \ \ p\geq 0, \ \ q \geq 0.
\end{aligned}
\end{equation}
Problem~\eqref{subprob:stationarity-qp} is a convex QP that we solve using the dual active-set QP solver in Gurobi. In Algorithm~\ref{alg:main}, the proximal parameter $\alpha_k$ remains unchanged, i.e., $\alpha_{k+1} \gets \alpha_k$ (Line~\ref{line:alpha-same}), whenever the sufficient decreasing condition at Line~\ref{line:check-suff-decrease} is satisfied; in our implementation, we instead update it as $\alpha_{k+1} \gets \max\{\xi^{-1}\alpha_k,10 \}$, which allows the proximal parameter to possibly take larger values.  We found this update strategy to work better in our testing, all of the analysis of Section~\ref{sec:limit-licq} still holds, and the analysis of Section~\ref{sec:sequential-licq} still holds if this modified update is only allowed a finite (possibly large) number of times.

The parameters used and initial proximal parameter value are presented in Table~\ref{tab:alg-params}.  The starting point $x_0$ and initial proximal-parameter value $\alpha_0$ used for the test problems  will be specified in Section~\ref{subsec:cutest}--\ref{subsec:scca}. 

\begin{table}[ht]
\caption{Parameters used by Algorithm~\ref{alg:main}. Recall that $\kappa_v^\infty$ appears in~\eqref{subprob:reformulate-trust-region}.}
\label{tab:alg-params}
\centering
    \begin{tabular}{ccccccccc}
    \hline
    $\tau_{-1}$ & $\kappa_v$ & 
    $\kappa_v^{\infty}$ & $\sigma_c$ & $\epsilon_{\tau}$ & $\xi$ & $\gamma$ & $\eta_{\Phi}$ & $\eta_m$ \\
    \hline
    1 & $10^3$ & $10^{-2}$ & 0.1 & 0.1 & 0.5 & 0.5 & $10^{-4}$ & $10^{-4}$  \\
    \hline
    \end{tabular}
\end{table}

 Algorithm~\ref{alg:main} is terminated when one of the following conditions is satisfied.
\begin{itemize}
\item \textbf{Approximate KKT point.}   Algorithm~\ref{alg:main} is terminated during the $k$th iteration with $x_k$ considered an approximate KKT point if $\|c_k\|_2 \leq 10^{-6}$,  $\|g_k+g_{r,k} + J_k^T y_k + z_k\|_2 \leq 10^{-4}$, and $\|\min \{x_k, -z_k\}\|_2 \leq 10^{-4}$.
\item \textbf{Time limit.} Algorithm~\ref{alg:main} is terminated if the running time exceeds $1$ hour.
\end{itemize}

As is common in the literature, we scale the problem functions.  In particular, the objective and its gradient are scaled by the scaling factor
\begin{equation}\label{eq:scaling}
\text{scale\_factor} =
\begin{cases}
\dfrac{100}{\|\nabla f(x_0)\|_\infty} & \text{if $\|\nabla f(x_0)\|_\infty > 100$,} \\
\qquad 1 & \text{otherwise}.
\end{cases}
\end{equation}
A similar scaling strategy is applied to each constraint $c_i$ for $1\in[m]$.

For comparison, we consider the solver  Bazinga,\footnote{The code package of Bazinga is downloaded from~\url{https://github.com/aldma/Bazinga.jl}} which is a safeguarded augmented Lagrangian method and, to the best of our knowledge, the only open source code that can solve problem~\eqref{prob:general}; see~\cite{de2023constrained} for more details. The Bazinga algorithm is terminated when one of the following conditions is satisfied.
\begin{itemize}
\item \textbf{Approximate KKT point.} Bazinga is terminated if a certain primal feasibility and dual stationarity measure are less than $10^{-6}$.
\item \textbf{Not a number.} Bazinga is terminated if a NaN occurs. 
\item \textbf{Time limit.} Algorithm~\ref{alg:main} is terminated if the running time exceeds 1 hour.
\end{itemize}

\subsection{CUTEst test problems}\label{subsec:cutest} 

We first conduct experiments on a subset of the CUTEst test problems.  Given the objective function $f$, equality constraint $c_E(x) = 0$, inequality constraints $c_l \leq c_I(x) \leq c_u$ for some constant vectors $c_l$ and $c_u$, and bound constraints $b_l \leq x \leq b_u$ for some constant vectors $b_l$ and $b_u$ all supplied by CUTEst for a given test problem, we solve the $\ell_1$-regularized optimization problem
\begin{equation}~\label{prob:cutest-modified}
\begin{aligned}
\min_{(x,s,a) \in \mathbb{R}^{n+m_I+m}} 
\ f(x) + \lambda \|a\|_1 \
\text{s.t.} \ 
\begin{bmatrix}
c_E(x) \\
c_I(x) - s
\end{bmatrix}
+ a = 0, \ \begin{bmatrix}
b_l \\
c_l
\end{bmatrix} \leq \begin{bmatrix}
x \\
s
\end{bmatrix} \leq \begin{bmatrix}
b_u \\
c_u
\end{bmatrix},
\end{aligned}
\end{equation}
where $m_I$ is the number of inequality constraints and $\lambda \in \mathbb{R}_{>0}$ is a regularization parameter. The slack vector $s$ is introduced to reformulate inequality constraints as equality constraints plus bound constraints.  The vector $a$ is introduced in this manner so that we can control its sparsity for illustrative purposes in our experiments.


The subset of CUTEst problems were chosen based on the following selection criteria: (i) the objective function is not constant; (ii) the number of variables and constraints satisfy $1 \leq m \leq n \leq 100$; (iii) the total number of inequality constraints satisfies $m_I \geq 1$. For the choice of $\lambda$, we consider the following optimization problem
\begin{equation} \label{prob:cutest}
\begin{aligned}
\min_{x \in \mathbb{R}^n,\, s \in \mathbb{R}^{m_I}} \quad & f(x) \quad \text{s.t.} \quad
\begin{bmatrix}
c_E(x) \\
c_I(x) - s
\end{bmatrix}
= 0, \quad \begin{bmatrix}
b_l \\
c_l
\end{bmatrix} \leq \begin{bmatrix}
x \\
s
\end{bmatrix} \leq \begin{bmatrix}
b_u \\
c_u
\end{bmatrix},
\end{aligned}
\end{equation}
and let $(\xbar, \bar{s})$ be a first-order KKT point of this problem with Lagrange multiplier $y_{\text{eq}}$ associated with the equality constraints. Then, if $\lambda \geq \|y_{\text{eq}}\|_{\infty}$, the point $(\xbar, \bar{s}, 0)$ is a first-order KKT point for the optimization problem~\eqref{prob:cutest-modified}. With this observation, we set $\lambda = \|y_{eq}\|_{\infty} + 10$ where $y_{eq}$ is computed by solving problem~\eqref{prob:cutest} using IPOPT~\cite{wachter2006implementation}. Problems that are not successfully solved by IPOPT are removed from the test problems.  The final subset consisted of 81 CUTEst test problems.



For our tests, we set $\alpha_0 = 10$ and $x_0$ as the initial point supplied by CUTEst.

We compare the performance of Algorithm~\ref{alg:main} and Bazinga using several metrics; the results of our tests can be found in Table~\ref{tab:comparison-table-modified-cutest-Bazinga}.  The meaning of the columns found in  Table~\ref{tab:comparison-table-modified-cutest-Bazinga} are described in the following bullet points.
\begin{itemize}  
\item \textbf{Feasible.}
The number of test problems for which the corresponding method terminates at a point with constraint violation less than $10^{-6}$. For this metric, we see that the two methods behave similarly, with Algorithm~\ref{alg:main} achieving approximate feasibility on four more test problem.
\item \textbf{Feasible, Better Objective.}
To understand the meaning of this column, let $\fus$ denote the final objective value returned by Algorithm~\ref{alg:main} and $\fBazinga$ denote the final objective value returned by Bazinga.  We then define the relative difference in the returned objective function values as
\begin{equation}~\label{eq:relative-diff-obj}
f_{\text{diff}} 
:=\frac{\fBazinga - \fus}{\max(1,|\min(\fBazinga,\fus)|)}.
\end{equation} 
We say that  Algorithm~\ref{alg:main} (resp., Bazinga) has a better relative objective value if $f_{\text{diff}} \geq 10^{-6}$ (resp., $f_{\text{diff}} \leq -10^{-6}$). Using this terminology, column ``Feasible, Better Objective'' gives the number of test problems for which both algorithms terminated at a point with constraint violation less than $10^{-6}$ \textit{and} the corresponding method has a better relative objective value. For this metric,  Algorithm~\ref{alg:main} outperforms Bazinga on $8$ additional problems.
\item \textbf{Performs Better}.
The number of test problems for which the corresponding method either (i) meets the constraint violation tolerance and the other method does not, or (ii) both methods reach the constraint violation tolerance and the corresponding method has a better relative objective value (see~\eqref{eq:relative-diff-obj}). For this metric, Algorithm~\ref{alg:main} outperforms Bazinga by one problem.
\item $a$ \textbf{is Zero.}
The number of test problems for which the  corresponding method returns $a=0$. Algorithm~\ref{alg:main} outperforms Bazinga on this metric, with Algorithm~\ref{alg:main} (resp., Bazinga) returning $a = 0$ on $76$ (resp., $55$) of the problems.

\item $a$ \textbf{is Small.}
The number of test problems for which the corresponding method returns $\|a\|_\infty \leq 10^{-8}$, thus indicating that $a$ is small (possibly equal to zero). When comparing this column with column ``$a$ is Zero'', we see that 
the only difference is that Bazinga returns a small (nonzero) value for $a$ on one additional test problem; the results for Algorithm~\ref{alg:main} are unchanged. 
\item \textbf{KKT Found.}
The number of test problems for which the corresponding method terminates with an approximate KKT point. Algorithm~\ref{alg:main} outperforms Bazinga with Algorithm~\ref{alg:main} (resp., Bazinga) returning an approximate first-order KKT point on $70$ (resp., $58$) of the problems tested.
\end{itemize}  

\begin{table}[ht!]
\label{tab:comparison-table-modified-cutest-Bazinga}
\centering
\caption{Algorithm~\ref{alg:main} versus Bazinga on various performance metrics related to solving problem~\eqref{prob:cutest-modified}.}
\resizebox{\textwidth}{!}{
\begin{tabular}{|c|c|c|c|c|c|c|c|}
\hline\rule{0pt}{0.8\normalbaselineskip}
Method & Feasible & Feasible, & Performs & $a$ is & 
$a$ is & KKT \\
       &          & Better Objective & Better & Zero  & Small & Found \\
\hline\hline
Algorithm~\ref{alg:main} & 	71 & 13 & 14 & 76 & 76 & 70 \\
\hline
Bazinga & 67 & 5 & 13 & 55 & 56 & 58 \\
\hline
\end{tabular}
}
\end{table}

We conclude this section by comparing the computational times of Algorithm~\ref{alg:main} and Bazinga. Figure~\ref{fig:performance-profile} is a Dolan-Moré performance profile~\cite{dolan2002benchmarking} for timings,  capped at $t=1000$. The results show that Algorithm~\ref{alg:main} (red line) outperforms Bazinga (purple line); see~\cite{dolan2002benchmarking} for details on interpreting this figure.

\begin{figure}[ht!]
\centering
\label{fig:performance-profile}
\includegraphics[width=0.7\linewidth]{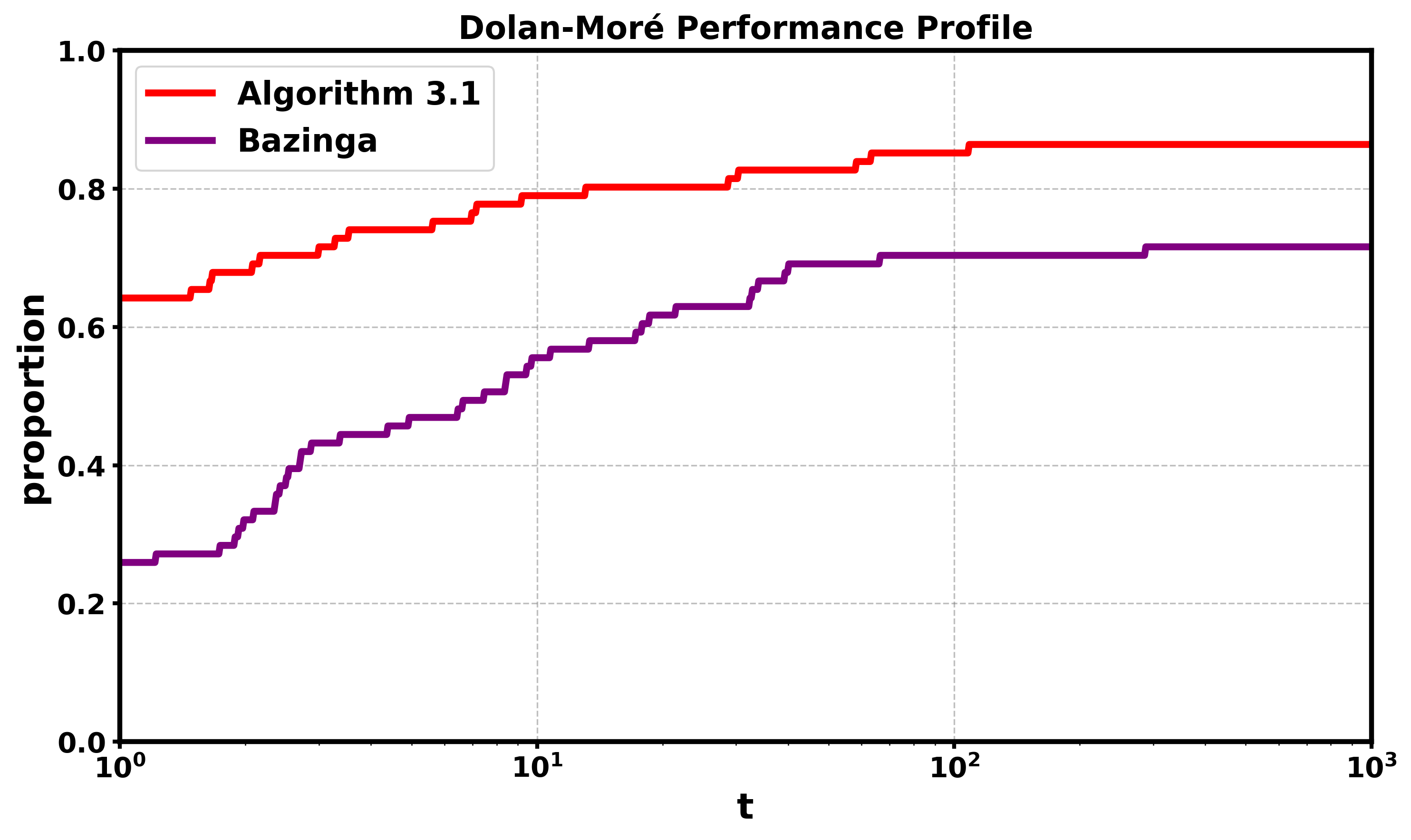}
\caption{More-Dolen performance profile comparing  Algorithm~\ref{alg:main} and Bazinga in terms of wall-clock time on the subset of CUTEst test problems discussed in Section~\ref{subsec:cutest}.}
\end{figure}

\subsection{Sparse canonical correlation analysis (SCCA)}\label{subsec:scca}

We now evaluate the performance of Algorithm~\ref{alg:main} on the SCCA problem~\cite{witten2009penalized} formulated as 
\begin{equation}~\label{prob:scca}
\begin{aligned}
\min_{w_x \in \mathbb{R}^{n_x}, w_y \in \mathbb{R}^{n_y}} & \quad - w_x^T \Sigma_{xy} w_y + \lambda (\|w_x\|_1 + \|w_y\|_1) \\
\text{s.t.} \quad \quad & \quad  w_x^T \Sigma_{xx} w_x \leq 1, \quad w_y^T \Sigma_{yy} w_y \leq 1,
\end{aligned}
\end{equation}
where $\Sigma_{xx} = XX^T$ and $\Sigma_{yy} = YY^T$ represent the covariance matrices for data matrices $X \in \mathbb{R}^{n_x \times N}$ and $Y \in \mathbb{R}^{n_y \times N}$, respectively, and $\Sigma_{xy} = XY^T$ represents the cross-covariance matrix between $X$ and $Y$. Problem~\eqref{prob:scca} aims to identify sparse weight vectors $w_x$ and $w_y$ that maximize the correlation between the transformed views of $X$ and $Y$ while the variance constraints prevent trivial solutions where the weight vectors are arbitrarily scaled to inflate the correlation.

Following the approach of~\cite{chu2013sparse}, we generate synthetic data matrices $X$ and $Y$ as
\begin{equation*}
X = \left( 
  \begin{bmatrix}
  \phantom{-}e \\ -e \\
  \phantom{-} 0
  \end{bmatrix}
  + \xi_x \right) u^T \quad \text{and} \quad Y = \left(
  \begin{bmatrix}
  \phantom{-}0 \\
  \phantom{-}e \\
  -e
  \end{bmatrix}
  + \xi_y \right) u^T,
\end{equation*}
where $e \in \mathbb{R}^{n_x/8}$ represents an all-ones vector, $\xi_x \in \mathbb{R}^{n_x}$ and $\xi_y \in \mathbb{R}^{n_y}$ are noise vectors with entries sampled from $\mathcal{N}(0, 0.01)$, and $u \in \R{N}$ is a random vector with entries $u_i \sim \mathcal{N}(0, 1)$. This construction creates a known ground truth structure: the first $n_x/4$ rows of $X$ are correlated with the last $n_y/4$ rows of $Y$. Consequently, the ideal sparse solutions for $w_x$ and $w_y$ should have non-zero elements confined to the first $n_x/4$ and last $n_y/4$ indices, respectively.

To evaluate the quality of a solution returned by a solver, we compute various metrics: the  correlation coefficient $\rho_{xy}$, sparsity ratio $sr_x$ for vector $w_x$, sparsity ratio $sr_y$ for vector $w_y$, overall sparsity ratio $sr$,
variance bound-constraint violations $voc_x$ and $voc_y$, and sparsity level $sl$, which are defined as
\begin{alignat*}{2}
\rho_{xy} &= \frac{ w_x^T \Sigma_{xy} w_y }{\sqrt{ (w_x^T \Sigma_{xx} w_x) (w_y^T \Sigma_{yy} w_y)}}, &
sr_x &= \frac{n_x - \|w_x\|_0}{n_x}, \\
sr_y &= \frac{n_y - \|w_y\|_0}{n_y}, &
sr &= \frac{(n_x + n_y) - (\|w_x\|_0 + \|w_y\|_0)}{n_x + n_y}, \\
voc_x &= \max\left( w_x^T \Sigma_{xx} w_x - 1, 0\right), &
voc_y &= \max\left( w_y^T \Sigma_{yy} w_y - 1, 0\right), \ \text{and} \\
&& sl &= \|[w_x]_{[n_x/4+1:n_x]}\|_0 + \|[w_y]_{[1:3n_y/4-1]}\|_0. 
\end{alignat*}



We consider SCCA test problems of three different sizes with $n_x = n_y = N \in \{200, 400, 800\}$ and regularization parameters $\lambda \in \{10^{-2}, 10^{-3}, 10^{-4}\}$. For each problem instance, the starting point $x_0$ is obtained by solving the generic canonical correlation analysis problem (no regularization term) using the \texttt{CCA} class from the \texttt{scikit-learn} package. We set the initial proximal parameter as  $\alpha_0 = 10^{-3}$. The algorithm terminates when one of the conditions detailed in Section~\ref{subsec:implementation} is satisfied.

\begin{table}[ht]
\caption{Performance metrics for Algorithm~\ref{alg:main} when solving problem~\eqref{prob:scca}.  Time is measured in seconds.}
\label{tab:scca.results}
\centering
\resizebox{\textwidth}{!}{%
\begin{tabular}{c|c|cccccccc}
\hline
$n_x =n_y$ & $\lambda$ & $\rho_{xy}$ & $sr_x$ & $sr_y$ & $sr$ & $sl$ & $voc_x$ & $voc_y$ & time \\
\hline
\multirow{3}{*}{200} & $10^{-2}$ & 1.0000 & 99.50\% & 99.50\% & 99.50\% & $0$ & 0 & 0 & 76.89 \\
& $10^{-3}$ & 1.0000 & 99.50\% & 99.50\% & 99.50\% &$0$ & 0 & 0 & 87.36 \\
& $10^{-4}$ & 1.0000 & 89.50\% & 90.00\% & 89.75\% & $0$ & 0 & 1.03e-11 & 117.14 \\
\hline
\multirow{3}{*}{400} & $10^{-2}$ & 1.0000 & 99.75\% & 99.75\% & 99.75\% & $0$ & 1.40e-9 & 0 & 128.40 \\
& $10^{-3}$ & 1.0000 & 99.50\% & 99.00\% & 99.25\% & $0$ & 9.83e-11 & 0 & 348.44 \\
& $10^{-4}$ & 1.0000 & 83.50\% & 82.75\% & 83.13\% & $0$ & 9.46e-11 & 1.67e-10 & 226.48 \\
\hline
\multirow{3}{*}{800} & $10^{-2}$ & 1.0000 & 99.88\% & 99.88\% & 99.88\% & $0$ & 5.86e-9 & 3.34e-9 & 279.18\\
& $10^{-3}$ & 1.0000 & 99.63\% & 99.88\% & 99.75\% & $0$ & 6.33e-10 & 1.81e-9 & 899.06 \\
& $10^{-4}$ & 1.0000 & 96.63\% & 95.63\% & 96.13\% & $0$ & 0 & 1.47e-10 & 463.84 \\
\hline
\end{tabular}%
}
\end{table}

\begin{table}[ht]
\caption{Performance metrics for Bazinga when solving problem~\eqref{prob:scca}. Time is measured in seconds.}
\label{tab:scca.bazinga.results}
\centering
\resizebox{\textwidth}{!}{
\begin{tabular}{c|c|cccccccc}
\hline
$n_x =n_y$ & $\lambda$ & $\rho_{xy}$ & $sr_x$ & $sr_y$ & $sr$ & $sl$ & $voc_x$ & $voc_y$ & \text{time} \\
\hline
\multirow{3}{*}{200} & $10^{-2}$ & 1.0000 & 99.50\% & 99.50\% & 99.50\% & $0$ & 4.02e-9 & 3.34e-8 & 86.10 \\
& $10^{-3}$ & 1.0000 & 99.50\% & 99.50\% & 99.50\% & $0$ & 1.96e-8 & 0 & 251.97 \\
& $10^{-4}$ & 1.0000 & 92.00\% & 87.50\% & 89.75\% & $0$ & 0 & 0 & 164.08 \\
\hline
\multirow{3}{*}{400} & $10^{-2}$ & 1.0000 & 99.75\% & 99.75\% & 99.75\% & $0$ & 6.62e-9 & 1.32e-8 & 556.60 \\
& $10^{-3}$ & 1.0000 & 97.50\% & 97.75\% & 97.63\% & $0$ & 0 & 0 & 744.31 \\
& $10^{-4}$ & 1.0000 & 77.75\% & 85.00\% & 81.38\% & $0$ & 0 & 0 & 713.13 \\
\hline
\multirow{3}{*}{800} & $10^{-2}$ & 1.0000 & 98.75\% & 98.38\% & 98.56\% & $0$ & 0 & 2.35e-9 & 2958.89 \\
& $10^{-3}$ & 1.0000 & 88.63\% & 97.25\% & 92.94\% & $0$ & 0 & 2.00e-8 & 2789.95 \\
& $10^{-4}$ & 1.0000 & 81.38\% & 78.75\% & 80.06\% & $0$ & 6.55e-8 & 0 & 2612.26 \\
\hline
\end{tabular}
}
\end{table}

The results in Table~\ref{tab:scca.results} demonstrate the effectiveness of Algorithm~\ref{alg:main} on SCCA problems. First, the correlation coefficient achieves the maximum possible value on every test case. Second, every solution exhibits the correct sparse structure since $sl = 0$. Third, the algorithm produces solutions with varying sparsity levels that are controlled by the regularization parameter $\lambda$, with higher sparsity ratios achieved by larger $\lambda$ values. Finally, constraint violations are smaller than $10^{-9}$. Table~\ref{tab:scca.bazinga.results} reports the performance of Bazinga on the same problems. Notably, Algorithm~\ref{alg:main} attains sparsity ratios that are at least as high as those of Bazinga (sometimes strictly higher), while requiring less computational time.

\section{Conclusion}\label{sec:conclusion}
We presented the first proximal-gradient–type method for regularized optimization problems with general nonlinear inequality constraints. Similar to the traditional proximal-gradient method, we proved that our approach has a convergence result (under an LICQ assumption), a worst-case iteration complexity result (under a stronger assumption), as well as a manifold identification property and active-set identification property (under standard assumptions).





\begin{thebibliography}{10}

\bibitem{bai2016splitting}
Yanqin Bai, Renli Liang, and Zhouwang Yang.
\newblock Splitting augmented {L}agrangian method for optimization problems
  with a cardinality constraint and semicontinuous variables.
\newblock {\em Optimization Methods and Software}, 31(5):1089--1109, 2016.

\bibitem{bareilles2023newton}
Gilles Bareilles, Franck Iutzeler, and J{\'e}r{\^o}me Malick.
\newblock Newton acceleration on manifolds identified by proximal gradient
  methods.
\newblock {\em Mathematical Programming}, 200(1):37--70, 2023.

\bibitem{BecT09}
A.~Beck and M.~Teboulle.
\newblock A fast iterative shrinkage-thresholding algorithm for linear inverse
  problems.
\newblock {\em SIAM Journal on Imaging Sciences}, 2(1):183--202, 2009.

\bibitem{beck2017first}
Amir Beck.
\newblock {\em First-order methods in optimization}.
\newblock SIAM, 2017.

\bibitem{bertocchi2020deep}
Carla Bertocchi, Emilie Chouzenoux, Marie-Caroline Corbineau, Jean-Christophe
  Pesquet, and Marco Prato.
\newblock Deep unfolding of a proximal interior point method for image
  restoration.
\newblock {\em Inverse Problems}, 36(3):034005, 2020.

\bibitem{bertsekas2009convex}
Dimitri Bertsekas.
\newblock {\em Convex optimization theory}, volume~1.
\newblock Athena Scientific, 2009.

\bibitem{boob2025level}
Digvijay Boob, Qi~Deng, and Guanghui Lan.
\newblock Level constrained first order methods for function constrained
  optimization.
\newblock {\em Mathematical Programming}, 209(1):1--61, 2025.

\bibitem{bourkhissi2025complexity}
Lahcen~El Bourkhissi, Ion Necoara, Panagiotis Patrinos, and Quoc Tran-Dinh.
\newblock Complexity of linearized perturbed augmented {L}agrangian methods for
  nonsmooth nonconvex optimization with nonlinear equality constraints.
\newblock {\em arXiv preprint arXiv:2503.01056}, 2025.

\bibitem{calamai1987projected}
Paul~H. Calamai and Jorge~J. Mor{\'e}.
\newblock Projected gradient methods for linearly constrained problems.
\newblock {\em Mathematical programming}, 39(1):93--116, 1987.

\bibitem{CheCR17}
Tianyi Chen, Frank~E. Curtis, and Daniel~P. Robinson.
\newblock A reduced-space algorithm for minimizing {$\ell_1$}-regularized
  convex functions.
\newblock {\em SIAM J. Optim.}, 27(3):1583--1610, 2017.

\bibitem{CheCR18}
Tianyi Chen, Frank~E. Curtis, and Daniel~P. Robinson.
\newblock {FaRSA} for {$\ell_1$}-regularized convex optimization: local
  convergence and numerical experience.
\newblock {\em Optim. Methods. Softw.}, 33(2):396--415, 2018.

\bibitem{chouzenoux2020proximal}
Emilie Chouzenoux, Marie-Caroline Corbineau, and Jean-Christophe Pesquet.
\newblock A proximal interior point algorithm with applications to image
  processing.
\newblock {\em Journal of Mathematical Imaging and Vision}, 62(6):919--940,
  2020.

\bibitem{chu2013sparse}
Delin Chu, Li-Zhi Liao, Michael~K Ng, and Xiaowei Zhang.
\newblock Sparse canonical correlation analysis: New formulation and algorithm.
\newblock {\em IEEE transactions on pattern analysis and machine intelligence},
  35(12):3050--3065, 2013.

\bibitem{cui2013convex}
XT~Cui, XJ~Zheng, SS~Zhu, and XL~Sun.
\newblock Convex relaxations and {MIQCQP} reformulations for a class of
  cardinality-constrained portfolio selection problems.
\newblock {\em Journal of Global Optimization}, 56(4):1409--1423, 2013.

\bibitem{CurtRobi25}
Frank~E. Curtis and Daniel~P. Robinson.
\newblock {\em Practical Nonconvex Nonsmooth Optimization}.
\newblock MOS-SIAM Series on Optimization. Society for Industrial and Applied
  Mathematics, Philadelphia, PA, USA, 2025.

\bibitem{dai2024proximal}
Yutong Dai, Xiaoyi Qu, and Daniel~P. Robinson.
\newblock A proximal-gradient method for equality constrained optimization.
\newblock {\em SIAM Journal on Optimization}, 35(4):2654--2683, 2025.

\bibitem{de2024interior}
Alberto De~Marchi and Andreas Themelis.
\newblock An interior proximal gradient method for nonconvex optimization.
\newblock {\em Open Journal of Mathematical Optimization}, 5:1--22, 2024.

\bibitem{de2023constrained}
Alberto De~Marchi, Xiaoxi Jia, Christian Kanzow, and Patrick Mehlitz.
\newblock Constrained composite optimization and augmented {L}agrangian
  methods.
\newblock {\em Mathematical Programming}, 201(1):863--896, 2023.

\bibitem{dolan2002benchmarking}
Elizabeth~D. Dolan and Jorge~J. Mor{\'e}.
\newblock Benchmarking optimization software with performance profiles.
\newblock {\em Mathematical programming}, 91(2):201--213, 2002.

\bibitem{dorfler2014sparsity}
Florian D{\"o}rfler, Mihailo~R Jovanovi{\'c}, Michael Chertkov, and Francesco
  Bullo.
\newblock Sparsity-promoting optimal wide-area control of power networks.
\newblock {\em IEEE Transactions on Power Systems}, 29(5):2281--2291, 2014.

\bibitem{facchinei1998accurate}
Francisco Facchinei, Andreas Fischer, and Christian Kanzow.
\newblock On the accurate identification of active constraints.
\newblock {\em SIAM Journal on Optimization}, 9(1):14--32, 1998.

\bibitem{fardad2011sparsity}
Makan Fardad, Fu~Lin, and Mihailo~R. Jovanovi{\'c}.
\newblock Sparsity-promoting optimal control for a class of distributed
  systems.
\newblock In {\em Proceedings of the 2011 American Control Conference}, pages
  2050--2055. IEEE, 2011.

\bibitem{gould2015cutest}
Nicholas~I.M. Gould, Dominique Orban, and Philippe~L. Toint.
\newblock {CUTEst}: a constrained and unconstrained testing environment with
  safe threads for mathematical optimization.
\newblock {\em Computational optimization and applications}, 60:545--557, 2015.

\bibitem{gurobi}
{Gurobi Optimization, LLC}.
\newblock {Gurobi Optimizer Reference Manual}, 2023.

\bibitem{hajinezhad2019perturbed}
Davood Hajinezhad and Mingyi Hong.
\newblock Perturbed proximal primal--dual algorithm for nonconvex nonsmooth
  optimization.
\newblock {\em Mathematical Programming}, 176(1):207--245, 2019.

\bibitem{hallak2023adaptive}
Nadav Hallak and Marc Teboulle.
\newblock An adaptive {L}agrangian-based scheme for nonconvex composite
  optimization.
\newblock {\em Mathematics of Operations Research}, 48(4):2337--2352, 2023.

\bibitem{hamza2019hybrid}
Syed~Ali Hamza and Moeness~G Amin.
\newblock Hybrid sparse array beamforming design for general rank signal
  models.
\newblock {\em IEEE Transactions on Signal Processing}, 67(24):6215--6226,
  2019.

\bibitem{han2015learning}
Song Han, Jeff Pool, John Tran, and William Dally.
\newblock Learning both weights and connections for efficient neural network.
\newblock {\em Advances in neural information processing systems}, 28, 2015.

\bibitem{hoefler2021sparsity}
Torsten Hoefler, Dan Alistarh, Tal Ben-Nun, Nikoli Dryden, and Alexandra Peste.
\newblock Sparsity in deep learning: Pruning and growth for efficient inference
  and training in neural networks.
\newblock {\em Journal of Machine Learning Research}, 22(241):1--124, 2021.

\bibitem{huang2023sparse}
Huiping Huang, Hing~Cheung So, and Abdelhak~M Zoubir.
\newblock Sparse array beamformer design via {ADMM}.
\newblock {\em IEEE Transactions on Signal Processing}, 71:3357--3372, 2023.

\bibitem{jiang2019structured}
Bo~Jiang, Tianyi Lin, Shiqian Ma, and Shuzhong Zhang.
\newblock Structured nonconvex and nonsmooth optimization: algorithms and
  iteration complexity analysis.
\newblock {\em Computational Optimization and Applications}, 72(1):115--157,
  2019.

\bibitem{KariNutiSchm16}
H.~Karimi, J.~Nutini, and M.~Schmidt.
\newblock Linear convergence of gradient and proximal-gradient methods under
  the {Polyak-Lojasiewicz} condition.
\newblock In {\em Joint European Conference on Machine Learning and Knowledge
  Discovery in Databases}, pages 795--811. Springer, 2016.

\bibitem{kong2019complexity}
Weiwei Kong, Jefferson~G. Melo, and Renato~D.C. Monteiro.
\newblock Complexity of a quadratic penalty accelerated inexact proximal point
  method for solving linearly constrained nonconvex composite programs.
\newblock {\em SIAM Journal on Optimization}, 29(4):2566--2593, 2019.

\bibitem{leconte2024interior}
Geoffroy Leconte and Dominique Orban.
\newblock An interior-point trust-region method for nonsmooth regularized
  bound-constrained optimization.
\newblock {\em arXiv preprint arXiv:2402.18423}, 2024.

\bibitem{lee2023accelerating}
Ching-pei Lee.
\newblock Accelerating inexact successive quadratic approximation for
  regularized optimization through manifold identification.
\newblock {\em Mathematical Programming}, 201(1):599--633, 2023.

\bibitem{lee2019inexact}
Ching-pei Lee and Stephen~J. Wright.
\newblock Inexact successive quadratic approximation for regularized
  optimization.
\newblock {\em Comput. Optim. Appl.}, 72:641--674, 2019.

\bibitem{lewis2013partial}
Adrian~S. Lewis and Shanshan Zhang.
\newblock Partial smoothness, tilt stability, and generalized {H}essians.
\newblock {\em SIAM Journal on Optimization}, 23(1):74--94, 2013.

\bibitem{li2021rate}
Zichong Li, Pin-Yu Chen, Sijia Liu, Songtao Lu, and Yangyang Xu.
\newblock Rate-improved inexact augmented {L}agrangian method for constrained
  nonconvex optimization.
\newblock In {\em International Conference on Artificial Intelligence and
  Statistics}, pages 2170--2178. PMLR, 2021.

\bibitem{liu2025proximal}
Shuai Liu, Claudia Sagastizabal, and Mikhail~V. Solodov.
\newblock Proximal gradient-method with superlinear convergence for nonsmooth
  convex optimization.
\newblock {\em SIAM Journal on Optimization}, 35(3):1601--1629, 2025.

\bibitem{mccormick1972gradient}
G.P. McCormick and R.A. Tapia.
\newblock The gradient projection method under mild differentiability
  conditions.
\newblock {\em SIAM Journal on Control}, 10(1):93--98, 1972.

\bibitem{more2006trust}
Jorge~J Mor{\'e}.
\newblock Trust regions and projected gradients.
\newblock In {\em System Modelling and Optimization: Proceedings of the 13th
  IFIP Conference Tokyo, Japan, August 31--September 4, 1987}, pages 1--13.
  Springer, 2006.

\bibitem{nutini2019active}
Julie Nutini, Mark Schmidt, and Warren Hare.
\newblock “{A}ctive-set complexity” of proximal gradient: How long does it
  take to find the sparsity pattern?
\newblock {\em Optimization Letters}, 13:645--655, 2019.

\bibitem{oberlin2006active}
Christina Oberlin and Stephen~J. Wright.
\newblock Active set identification in nonlinear programming.
\newblock {\em SIAM Journal on Optimization}, 17(2):577--605, 2006.

\bibitem{Rob07}
Daniel~P. Robinson.
\newblock {\em Primal-Dual Methods for Nonlinear Optimization}.
\newblock PhD thesis, Department of Mathematics, University of California San
  Diego, La Jolla, CA, 2007.

\bibitem{rockafellar2009variational}
R.~Tyrrell Rockafellar and Roger J-B Wets.
\newblock {\em Variational analysis}, volume 317.
\newblock Springer Science \& Business Media, 2009.

\bibitem{sahin2019inexact}
Mehmet~Fatih Sahin, Ahmet Alacaoglu, Fabian Latorre, Volkan Cevher, et~al.
\newblock An inexact augmented {L}agrangian framework for nonconvex
  optimization with nonlinear constraints.
\newblock {\em Advances in Neural Information Processing Systems}, 32, 2019.

\bibitem{sun2019we}
Yifan Sun, Halyun Jeong, Julie Nutini, and Mark Schmidt.
\newblock Are we there yet? manifold identification of gradient-related
  proximal methods.
\newblock In {\em The 22nd International Conference on Artificial Intelligence
  and Statistics}, pages 1110--1119. PMLR, 2019.

\bibitem{toint1988global}
Ph.~L. Toint.
\newblock Global convergence of a of trust-region methods for nonconvex
  minimization in hilbert space.
\newblock {\em IMA Journal of Numerical Analysis}, 8(2):231--252, 1988.

\bibitem{tonneau2020sl1m}
Steve Tonneau, Daeun Song, Pierre Fernbach, Nicolas Mansard, Michel Ta{\"\i}x,
  and Andrea Del~Prete.
\newblock {SL1M}: Sparse {L}1-norm minimization for contact planning on uneven
  terrain.
\newblock In {\em 2020 IEEE International Conference on Robotics and Automation
  (ICRA)}, pages 6604--6610. IEEE, 2020.

\bibitem{wachter2006implementation}
Andreas W{\"a}chter and Lorenz~T. Biegler.
\newblock On the implementation of an interior-point filter line-search
  algorithm for large-scale nonlinear programming.
\newblock {\em Mathematical programming}, 106(1):25--57, 2006.

\bibitem{wei2024sqp}
Pinzheng Wei and Weihong Yang.
\newblock An {SQP}-type proximal gradient method for composite optimization
  problems with equality constraints.
\newblock {\em Journal of Computatinal Mathematics}, 2024.

\bibitem{witten2009penalized}
Daniela~M Witten, Robert Tibshirani, and Trevor Hastie.
\newblock A penalized matrix decomposition, with applications to sparse
  principal components and canonical correlation analysis.
\newblock {\em Biostatistics}, 10(3):515--534, 2009.

\bibitem{yang2019survey}
Xinghao Yang, Weifeng Liu, Wei Liu, and Dacheng Tao.
\newblock A survey on canonical correlation analysis.
\newblock {\em IEEE Transactions on Knowledge and Data Engineering},
  33(6):2349--2368, 2019.

\bibitem{zhang2017global}
Yuqian Zhang, Yenson Lau, Han-wen Kuo, Sky Cheung, Abhay Pasupathy, and John
  Wright.
\newblock On the global geometry of sphere-constrained sparse blind
  deconvolution.
\newblock In {\em Proceedings of the IEEE Conference on Computer Vision and
  Pattern Recognition}, pages 4894--4902, 2017.

\bibitem{zou2018selective}
Hui Zou and Lingzhou Xue.
\newblock A selective overview of sparse principal component analysis.
\newblock {\em Proceedings of the IEEE}, 106(8):1311--1320, 2018.

\end{thebibliography}

\end{document}